\documentclass[10pt]{amsart}

\usepackage{texmaX,semmaX,semtkzX}
\usepackage{dsfont}
\usepackage{makecell}
\usepackage{multirow}
\usepackage{float}

\advance\oddsidemargin -1cm
\advance\evensidemargin -1cm
\advance\textwidth 2cm
\advance\textheight 0.3cm 

\def\mod{\quad\textup{mod }}
\def\JX{\Jac_X}

\def\cX{{\mathcal{X}}}

\theoremstyle{theorem}
\newtheorem*{conjecture*}{Conjecture}
\newtheorem{convention}{Convention}[section]
\newtheorem*{constructions}{Construction}

\begin{document}
\title[On Galois covers of curves and arithmetic of Jacobians]{On Galois covers of curves and arithmetic of Jacobians}

\author{Alexandros Konstantinou and Adam Morgan}

\address{University College London, London WC1H 0AY, UK}
\email{alexandros.konstantinou.16@ucl.ac.uk}

\address{University of Glasgow, Glasgow, G12 8QQ.}
\email{adam.morgan@glasgow.ac.uk}

\begin{abstract} 
We study the arithmetic of curves and Jacobians endowed with the action of a finite group $G$. This includes a study of the basic properties, as $G$-modules, of  their $\ell$-adic representations, Selmer groups, rational points and Shafarevich--Tate groups. In particular, we show that $p^\infty$-Selmer groups are self-dual $G$-modules, and give various `$G$-descent' results for Selmer groups and rational points. Along the way we revisit, and slightly refine, a construction going back to Kani and Rosen for associating isogenies  to homomorphisms between permutation representations.

With a view to future applications, it is convenient to work throughout with curves that are not assumed to be geometrically connected (or even connected); such curves arise naturally when taking Galois closures of covers of curves. For lack of a suitable reference, we carefully detail how to deduce the relevant properties of such curves and their Jacobians from the more standard geometrically connected case.  
\end{abstract}

\maketitle

\setcounter{tocdepth}{1}
\tableofcontents

\section{Introduction}

The aim of this paper is to establish some basic, yet fundamental, arithmetic properties of curves and their Jacobians in the presence of the action of a finite group of automorphisms. A key motivator for this work is to support applications to the parity conjecture for Jacobians, which appear in work of the authors with V. Dokchitser and Green \cite{DGKM}. A central theme of that work is that one can study the arithmetic of Jacobians by exploiting \textit{Brauer relations} in the context of function fields of curves; alongside a broader study of group actions on curves over arithmetic fields, this paper sets up the machinery required to do this. With these applications in mind, it will be important to work throughout with a slightly more general notion of `curve' than is perhaps standard. Specifically, our convention is as follows:

\medskip
\begin{convention} \label{convention}
Throughout this paper, curves are assumed to be smooth and proper, but are not assumed to be connected, and neither are their connected components assumed to be geometrically connected. 
\end{convention}
 \medskip
 
\noindent  Such curves emerge naturally in the context of Galois covers; see Example \ref{ex:S3part1} below.  For completeness, and for lack of a suitable reference, we dedicate part of this paper to carefully detailing the basic properties enjoyed by such curves and their Jacobians. For convenience when stating the other results in the paper, we turn to this first.

\subsection{General curves and their Jacobians}

Let $k$ be a field of characteristic $0$. Under Convention \ref{convention}, a curve $X$ over $k$ admits a decomposition as a disjoint union of connected, yet not necessarily geometrically connected, curves. Much of the more standard theory of geometrically connected curves and their Jacobians carries over seamlessly to this framework. In particular, the equivalence between connected curves and function fields generalises to one between curves and certain \'{e}tale algebras. Moreover, the associated Jacobian variety\footnote{Defined as the identity component of the relative Picard functor of $X/k$.} $\textup{Jac}_{X}$ is a principally polarised abelian variety, admitting a decomposition into a product of Weil restrictions of Jacobians of geometrically connected curves. More precisely, we have the following proposition, giving the basic properties we will need. 

\begin{proposition}[Proposition \ref{curves_vs_etale_algebras} \& Lemma \ref{lem:curve_Weil restriction}] \label{Prop:Curves-Algebras}

Let $k$ be a field of characteristic $0$, and let $X$ be a curve over  $k$ satisfying Convention \ref{convention}. Then we have:

\begin{enumerate}
\item the Jacobian $\textup{Jac}_{X}$ of $X$ is a principally polarised abelian variety over $k$, 
\item there exist finite field extensions $k_{i}/k$, and geometrically connected curves $\{X_i/k_i\}_{i=1}^{n}$, such that
\[\textup{Jac}_{X} \cong \prod_{i=1}^{n} \textup{Res}_{k_i/k}\textup{Jac}_{X_i},\]
where $\textup{Res}_{k_i/k}$ denotes Weil restriction from $k_i$ to $k$.
\item writing $F=k(Y)$ for the function field of a geometrically connected curve $Y/k$,  there is an equivalence of categories
\[\left\{\begin{array}{c}\textup{covers }X\to Y\\\textup{ and } \\ \textup{morphisms between covers}  \end{array}\right\}  \stackrel{ }{\longleftrightarrow}  \left\{\begin{array}{c}\textup{finite \'{e}tale }$F$\textup{-algebras} \\\textup{ and }\\ $F$\textup{-algebra inclusions} \end{array}\right\},\]
which is compatible with base change to arbitrary field extensions. Here by a \textit{cover} we mean a curve $X/k$ equipped with a finite surjective morphism $X \to Y$; a morphism of covers of $Y$ is a finite surjective $Y$-morphism.
\end{enumerate}
\end{proposition} 

We remark that since $k$ is assumed to have characteristic $0$, a finite \'{e}tale $F$-algebra is simply a finite product of finite field extensions of $k$. 

Part (3) of the proposition allows the construction of the $S_n$-closure of a degree $n$ cover $X\to Y$  (see Definition \ref{def:snclosurecurve}). This is a degree $n!$ cover $\tilde{X} \to Y$ that generalises the notion of the Galois closure of a branched cover of curves. This allows certain arguments to be carried out `uniformly', without worrying about the possibility that the Galois closure have degree smaller than $n!$. 

\subsection{Arithmetic of $G$-covers}
Having established  Proposition 1.1, the rest of the paper studies the arithmetic of finite group actions on the curves of Convention \ref{convention}. In particular, given a curve $X/K$ and a finite subgroup $G$ of $K$-automorphisms, we study various auxiliary $G$-representations. These include the $\ell$-adic Tate modules $V_{\ell}(\textup{Jac}{X})=T_\ell(\JX)\otimes\Q_\ell$ and, when $K$ is a number field, the rational points $\textup{Jac}_{X}(K)\otimes \mathbb{Q}$ and dual $p^\infty$-Selmer groups
 \[\mathcal{X}_{p}(\JX):= \textup{Hom}_{\mathbb{Z}_{p}}\big(\lim_\rightarrow\textup{Sel}_{p^n}(\JX), \mathbb{Q}_{p}/\mathbb{Z}_{p}\big) \otimes_{\mathbb{Z}_{p}} \mathbb{Q}_{p}.\]  
  Our main result for Selmer groups is the following.

\begin{theorem}[Theorem \ref{Theorem:Self-duality_selmer}] \label{self_duality_selmer} 
The $\mathbb{Q}_{p}[G]$-representation $\mathcal{X}_{p}(\textup{Jac}_{X})$ is self-dual.
\end{theorem}

This theorem is a key ingredient for investigating Selmer rank parities using the theory of \textit{regulator constants} (see \cite{tamroot}).
We remark that an immediate consequence of finiteness of the Shafarevich--Tate group of $\textup{Jac}_X$ would be that $\mathcal{X}_{p}(\textup{Jac}_{X})$ is in fact a \textit{rational} $G$-representation (and isomorphic to $\textup{Jac}_{X}(K)\otimes_{\mathbb{Z}} \mathbb{Q}_p$). However, without \textit{a priori} knowledge of the \hbox{Shafarevich--Tate} group,  even the self-duality of  $\mathcal{X}_{p}(\textup{Jac}_{X})$ requires careful argument.

We prove Theorem \ref{self_duality_selmer} more generally for Selmer groups of principally polarised abelian varieties equipped with the action of a finite group of automorphisms. This includes as a special case the self-duality of the Selmer group as a Galois module proven by T. and V. Dokchitser in \cite{dokchitser_dokchitser_2009}.  The broader scope provided by Theorem \ref{self_duality_selmer} finds applications in the context of Galois covers of curves in \cite{DGKM, AK}.

We now compare the arithmetic of $\textup{Jac}_X$ to the arithmetic of the Jacobian of the quotient curve $X/G$.   In this context, we confirm ``\textit{Galois descent}'' in various settings. 

\begin{theorem}[Theorem \ref{Galois_descent_rational_points} \& Proposition \ref{prop:heights_rescaling}]\label{thm:introgaloisdescent}
Let $X/K$ be a curve over a number field, and let $G$ be a finite group of $K$-automorphisms of $X$. Then we have canonical isomorphisms 
\begin{enumerate} 
  \item $(\Jac_X(K)\otimes\Q)^G \simeq \Jac_{X/G}(K)\otimes\Q$, 
  \item $\cX_p(\Jac_X)^G \simeq \cX_p(\Jac_{X/G})$,
  \item $V_{\ell}(\Jac_X)^G \simeq V_{\ell}(\Jac_{X/G})$.
\end{enumerate}
Moreover, under the isomorphism of (1), the N\'eron--Tate heights on $\Jac_X$ and $\Jac_{X/G}$ satisfy, for all $x,y$ in $\Jac_{X/G}(K)$, \[ \langle x,y \rangle_{\Jac_X}= |G|\cdot \langle x,y\rangle_{\Jac_{X/G}} .\]
\end{theorem}

Finally, we give a basic but useful criterion for demonstrating that certain representations of $G$ are absent from rational points and Selmer groups.

\begin{proposition}[Proposition \ref{Proposition: pseudo Brauer for Xp(A)}] \label{reps_absent}
In the setting of Theorem \ref{thm:introgaloisdescent}, suppose that an irreducible $G$-representation $\rho$ does not appear in the $\ell$-adic Tate module $V_\ell(\JX)$. Then it does not appear in rational points or in $p^\infty$-Selmer groups: 
$$
\langle \rho,V_\ell(\JX) \rangle =0 \implies
\langle\rho, \JX(K)\otimes\Q \rangle =\langle\rho, \cX_p(\JX) \rangle =0. 
$$
(Here $\left \langle ~,~\right \rangle$ denotes the usual inner product on characters of $G$-representations.)
\end{proposition}

\begin{remark}
As with Theorem \ref{self_duality_selmer}, Theorem \ref{thm:introgaloisdescent} and Proposition \ref{reps_absent} (suitably formulated) hold in the more general setting of abelian varieties equipped with a finite group action. Instances where such generalisations are feasible are highlighted in the main text. 
\end{remark}

\subsection{Homomorphisms from $G$-maps}
A key tool for studying the arithmetic of Jacobians in the presence of the action of a finite group $G$ is a formalism, originating in work of Kani and Rosen \cite{MR1000113}, for constructing isogenies between Jacobians from certain homomorphisms between $G$-modules. In this paper, we both extend this construction to the curves of Convention \ref{convention}, and give a refinement of the construction. Our result here is the following, a more detailed version of which appears as Theorem \ref{main_G_mod_corresp_thm}. In the statement, $H_{i}, H_{j}'$ denote subgroups of $G$.

\begin{theorem}[Theorem \ref{main_G_mod_corresp_thm}] \label{kani-rosen} 
Let $k$ be a field of characteristic $0$, and let $X/k$ be a curve and $G$ a finite group of automorphisms of $X$. Let $S=\coprod_{i} G/H_{i}$ and $S'=\coprod_{j} G/H_{j}'$ be two finite $G$-sets. Given a $\mathbb{Z}[G]$-module homomorphism \hbox{$\phi: \mathbb{Z}[S] \to \mathbb{Z}[S']$}, there is a natural $k$-homomorphism
\[f_{\phi}: \prod_{j} \textup{Jac}_{X/H_j'} \to \prod_{i} \textup{Jac}_{X/H_i}. \] The association $\phi \mapsto f_\phi$ is functorial in $\phi$. Moreover, the transposed matrix of $\phi$ (with respect to the standard bases), denoted $\phi^{\vee}$, satisfies $f_{\phi^{\vee}} = (f_{\phi})^{\vee}$. That is, the dual of $f_{\phi}$  (taken with respect to the canonical principal polarisations) agrees with $f_{\phi^{\vee}}$.
\end{theorem}

In addition, Theorem \ref{main_G_mod_corresp_thm} specifies an explicit criterion for when the resulting homomorpism $f_\phi$ is an isogeny. This is the case, for example, when $\phi$ induces an isomorphism of $\mathbb{Q}[G]$-representations $\mathbb{Q}[S] \cong \mathbb{Q}[S']$. Such isomorphisms can exist even when the underlying $G$-sets $S$ and $S'$ are not isomorphic; see \cite[Definition 2.1]{tamroot}. Such isomorphisms, known as \textit{Brauer relations}, provide a rich source of isogenies between Jacobians. 

\begin{remark}
The work of Kani--Rosen mentioned above verifies the \textit{existence} of an isogeny $f_\phi$ in the presence of suitable maps between permutation representations of $G$. However, their work does not give an explicit construction of the isogeny. In later work, Chen  \cite{MR1779608} gives an explicit construction, though the resulting map between abelian varieties does not (quite) satisfy the functoriality and duality results that our construction does. By contrast, work of de Smit and Edixhoven \cite{MR1000113} constructs the same map $f_\phi$ as us, but does so only  in the category of abelian varieties with isogenies inverted. See Remark \ref{rem:desmit_chen_remark} for a more precise comparison between those works and Theorem \ref{kani-rosen}.
\end{remark}

\subsection{Equivariant Riemann--Hurwitz}
Let $Y/k$ be a geometrically connected curve of genus $g(Y)$, and let $\pi:X \to Y$  be a $G$-cover of curves (that is, there is an isomorphism $Y\cong X/G$ making the obvious diagram commute). In this setting, we give a description of the $G$-module structure of the rational $\ell$-adic Tate module $V_{\ell}(\textup{Jac}_{X})$ in terms of permutation modules.  The case when $X$ is geometrically connected is the equivariant Riemann--Hurwitz formula, as can be found, for example, in  \cite[Proposition 1.1]{Riemann-Hurwitz}.  In the setting of the curves of Convention \ref{convention}, we verify the following extension of this result. 

\begin{proposition}[Proposition \ref{Theorem:equiv_RH}] \label{Thm:Equivariant Riemann--Hurwitz}
Let $\pi:X \to Y$  be a $G$-cover of curves. Let $\Gamma$ be a connected component of $X_{\bar{k}}$, and write $\pi|_{\Gamma}: \Gamma \to Y_{\bar{k}}$ for the restriction of (the base-change to $\bar{k}$ of) $\pi$ to $\Gamma$. Let $H\leq G$ denote the stabiliser of $\Gamma$, denote by $\{q_1,\dots,q_r\}\subseteq Y(\overline{k})$ the branch points of $\pi|_{\Gamma}$, and for each $1\leq i \leq r$, let $t_{i} \in \pi|_{\Gamma}^{-1}(q_i)$ be any choice of preimage $q_i$. Then, for every prime $\ell$, the $G$-representations
\begin{equation} \label{eqn:riemann_hurwitz}V_{\ell}(\textup{Jac}_{X}) \ \ \textup{and} \ \
\textup{Ind}_{H}^{G}(\mathds{1})^{\oplus 2} \oplus \textup{Ind}_{\{e\}}^{G}(\mathds{1})^{\oplus (r+2\textup{g}(Y)-2)}  \ominus \bigoplus_{i=1}^{r} \textup{Ind}_{\textup{Stab}_{H}(t_i)}^{G} \mathds{1}
  \end{equation} 
  are isomorphic after extending scalars to $\mathbb{C}$. (Here $\textup{Stab}_{H}(t_i)$ denotes the stabiliser of $t_i$ in $H$).
\end{proposition}

\subsection{Order of Shafarevich--Tate groups modulo squares}
Finally, we express the size, modulo rational squares, of the $2$-primary part $\Sha_{0}(\textup{Jac}_{X}/K)[2^{\infty}]$ of the quotient of the  Shafarevich--Tate group of $\textup{Jac}_{X}$ by its maximal divisible subgroup,  in terms of local invariants. This extends a result of Poonen--Stoll \cite{poonen1999cassels} from Jacobians of geometrically connected curves to our setting. The local invariant $\mu_v(X)$ in the statement is detailed in Definition \ref{def:deficiency}. It extends the notion of deficiency   given in \cite[Section 8]{poonen1999cassels}.  

\begin{proposition}[Proposition \ref{Lemma: deficiency for quasi-nice curves and 2 part of Sha}] Let $X$ be a curve defined over a number field $K$. Then, for each place $v$ of $K$, there exists a local invariant $\mu_{v}(X)$ such that
\[\#\sha_{0}(\JX)[2^{\infty}] \equiv \prod_{\textup{$v$ place of $K$}} \mu_v(X) \mod \mathbb{Q}^{\times 2}. \] 
\end{proposition}

\subsection{A worked example} \label{subsection:example} \label{ex:S3part1}
We now give a basic worked example illustrating several of the results above. Let $E/k$ be an elliptic curve given by a Weierstrass equation $y^2 = f(x)$ for a separable cubic $f(x)\in k[x]$. We view $E$ as a degree 3 cover of $\mathbb P^1$ via the $y$-coordinate map. The $S_3$-closure of $k(E)/k(y)$ has the following field diagram, where $D:\Delta^2 = G(y)$, $B: \{y^2 = f(x),\;\Delta^2 = G(y)\}$ and $G(y)\in k[y]$ is the discriminant of $f(x)-y^2$.

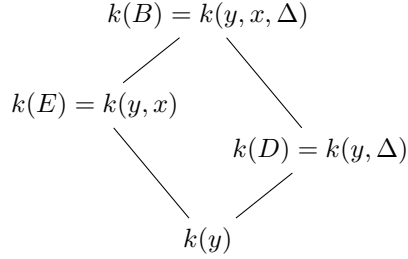
\begin{figure}[H]
\begin{center}\begin{tikzpicture}
    \node (Q1) at (0,0) {$k(y)$};
    \node (Q2) at (1.5,1.2) {$k(D) = k(y, \Delta)$};
    \node (Q3) at (-1.5,1.8) {$k(E) = k(y,x)$};
    \node (Q4) at (0,3) {$k(B) = k(y,x,\Delta)$};
    \draw (Q4)--(Q2);
    \draw (Q4)--(Q3);
    \draw (Q3)--(Q1);
    \draw (Q2)--(Q1);
    \end{tikzpicture} \end{center}\end{figure}
Consider the case of $k=\Q$ and $E:y^2=x^3+b$. 
The discriminant curve (the quotient of the $S_3$-closure by the action of $A_3$) becomes the normalisation of
$D: \Delta^2=-27(y^2-b)^2$, which is not absolutely irreducible. Over $\Q_7$, this is not even irreducible, and the corresponding function field $\Q_7(y, \Delta)$ should be viewed as the algebra $\Q_7(y)[\Delta]/(\Delta^2+27)$, rather than as the trivial extension of $\Q_7(y)$. 
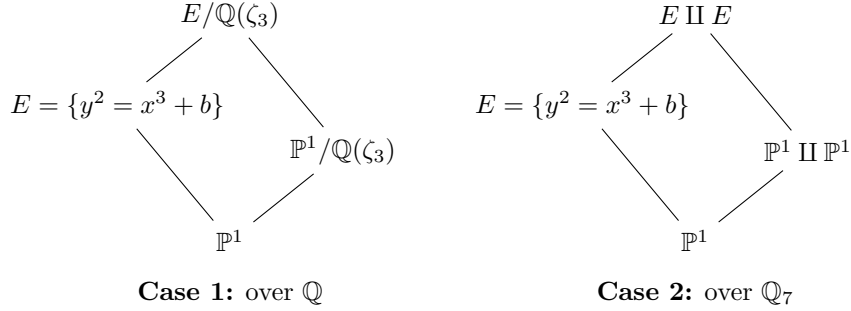
\begin{figure}[H]
\begin{tikzpicture}
    \node (Q1) at (0,0) {$\mathbb{P}^{1}$};
    \node (Q2) at (1.5,1.2) {$\mathbb{P}^{1}/\mathbb{Q}(\zeta_3)$};
    \node (Q3) at (-1.5,1.8) {$E=\{y^2=x^3+b\}$};
    \node (Q4) at (0,3) {$E/\mathbb{Q}(\zeta_3)$};
    \node (Q5) at (0,-0.7) {\textbf{Case 1:} over $\mathbb{Q}$};
    \draw (Q4)--(Q2);
    \draw (Q4)--(Q3);
    \draw (Q3)--(Q1);
    \draw (Q2)--(Q1); \end{tikzpicture} 
\qquad    
\begin{tikzpicture}
    \node (Q1) at (0,0) {$\mathbb{P}^{1}$};
    \node (Q2) at (1.5,1.2) {$\mathbb{P}^{1}\amalg \mathbb{P}^{1}$};
    \node (Q3) at (-1.5,1.8) {$E=\{y^2=x^3+b\}$};
    \node (Q4) at (0,3) {$E \amalg E$};
    \node (Q5) at (0,-0.7) {\textbf{Case 2:} over $\mathbb{Q}_7$};
    \draw (Q4)--(Q2);
    \draw (Q4)--(Q3);
    \draw (Q3)--(Q1);
    \draw (Q2)--(Q1); \end{tikzpicture} 
    \end{figure}
    
For the remainder of this section, we use Example \ref{ex:S3part1} above to illustrate applications of Theorems \ref{thm:introgaloisdescent} and \ref{kani-rosen}, and Proposition \ref{Thm:Equivariant Riemann--Hurwitz} . 

\subsubsection{\noindent {$G$-module structure of rational points and Selmer groups.}} \label{subsub_quotients} Suppose that $E$ is defined over a number field $K$. Then, $S_3$ acts on $B$ by $K$-automorphisms. We can then consider the decomposition $\Jac_B(K)\otimes\Q\cong\triv^{\oplus n}\oplus\epsilon^{\oplus m}\oplus \rho^{\oplus r}$ into irreducibles, where $n$, $m$, $r\in\N$ and $\triv $, $\epsilon$, $\rho$ denote, respectively, the trivial, sign and 2-dimensional irreducible representation of $S_3$. Applying Theorem \ref{thm:introgaloisdescent}(1) with $G = S_3$ and comparing dimensions gives $n=\rk\Jac_{\P^1}=0$. Repeating with $G = C_2$ and $C_3$ gives $\rk E= r$ and $\rk\Jac_D = m$. 
Arguing similarly, but instead applying part (2), we see that $\mathcal{X}_{p}(\Jac_B)\cong\epsilon^{\oplus \rk_p \Jac_D}\oplus \rho^{\oplus \rk_p E}$.

\subsubsection{Equivariant Riemann--Hurwitz} We consider the case of $k=\Q$ and $E:y^2=x^3+b$ as in \S\ref{ex:S3part1}.
To apply Proposition \ref{Thm:Equivariant Riemann--Hurwitz}, we consider the degree $3$ cyclic cover \hbox{$E/\mathbb{Q}(\zeta_3) \to \mathbb{P}^{1}/\mathbb{Q}(\zeta_3)$} given by $(x,y) \to x$. 
This is branched at $\{\infty, \sqrt{b}, - \sqrt{b}\}$ each with ramification degree $3$. By applying Proposition \ref{Thm:Equivariant Riemann--Hurwitz} with $G=S_3, H=C_3, r=3$ and $\textup{Stab}_{H}(t_i)=C_3$ for all $i$, we deduce that $V_{\ell}(\textup{Jac}_{B}) \cong \rho^{\oplus 2}.$

\subsubsection{Construction of isogenies} We now give an illustration of Theorem \ref{kani-rosen}. Fix generators $g,h$ such that $S_3 = \langle g,h \ | \ g^3=h^2=e, hg=g^2h \rangle$ and $k(E)=k(B)^{\langle h \rangle}$, $k(D) = k(B)^{\langle g \rangle}$. Consider the $\mathbb{Z}[S_3]$-module homomorphism
\begin{equation*}
\phi: \mathbb{Z}[S_3/\langle h \rangle ]x_1 \oplus \mathbb{Z}[S_3/\langle h \rangle ]x_2 \oplus \mathbb{Z}[S_3/\langle g \rangle ]x_3 \to \mathbb{Z}[S_3]y
\end{equation*}
determined by $x_1 \mapsto (1+h)y , x_2 \mapsto (g+hg)y , x_3 \mapsto (1+g+g^2)y$. With respect to the $\mathbb{Z}$-bases $\{x_1,gx_1,g^2x_1,x_2,gx_2,g^2x_2,x_3,hx_3\}$ and $\{y,gy,g^2y,hy,ghy,g^2hy\}$, $\phi$ is the matrix
\begin{equation*}
 \left[
\begin{array}{ccc|ccc|cc}
1 & 0 & 0 & 0 & 0 & 1 & 1 & 0 \\
0 & 1 & 0 & 1 & 0 & 0 & 1 & 0 \\
0 & 0 & 1 & 0 & 1 & 0 & 1 & 0 \\
1 & 0 & 0 & 0 & 0 & 1 & 0 & 1 \\
0 & 1 & 0 & 1 & 0 & 0 & 0 & 1 \\
0 & 0 & 1 & 0 & 1 & 0 & 0 & 1 
\end{array}
\right] 
\end{equation*}
The induced homomorphism  $f_{\phi}: \textup{Jac}_{B} \to E \times E \times \textup{Jac}_{D}$ afforded by Theorem \ref{kani-rosen} is 
\[f_{\phi} = \big(\pi_{E_*} , \pi_{E_*} \circ g_{*} , \pi_{D_*}  \big),\]
where $\pi_E: B \to E$ and $\pi_D: B\to D$ denote the quotient maps.
Transposing $\phi$, we obtain a $\mathbb{Z}[G]$-module map determined by $\phi^{\vee}(y) =( x_1,g^2x_2,x_3)$. The induced homomorphism \[ f_{\phi^{\vee}}: E \times E \times \textup{Jac}_{D} \to \textup{Jac}_{B} \] is given by $f_{\phi^{\vee}}((z_1,z_2,z_3))=\pi_{E}^{*}z_1+g^{*}\pi_{E}^{*}z_2+\pi_{D}^{*}z_3,$ which coincides with the dual  of $f_{\phi}$ (with respect to the canonical principal polarisations on each Jacobian).

\subsection{Notation and conventions} \label{General_notation} 

Throughout the paper, we implicitly fix  an embedding $\Q_\ell \hookrightarrow{} \mathbb{C}$ for all rational primes $\ell$. We will also adopt the following notation:

\medskip

\begin{tabular}{p{0.135\textwidth}p{0.81\textwidth}}
$k$ & field of characteristic $0$\\
$X$ & a curve defined over a field (in the sense of Convention \ref{convention})\\
$X/G$ & quotient of the curve $X$ by a finite group $G\leq \Aut_k(X)$\\
$A$ & an abelian variety over a field\\
$\JX$ & the Jacobian variety of $X$ (as in \S \ref{sec:curves_and_jacobians_background}, \S\ref{sec:quasi_nice_jac}); it is an abelian variety equipped with a canonical principal polarisation
\\
$\mathcal{X}_{p}(A)$ & $\text{Hom}_{\mathbb{Z}_{p}}(\varinjlim \text{Sel}_{p^n}(A), \mathbb{Q}_{p}/\mathbb{Z}_{p})\otimes \mathbb{Q}_{p}$, the dual $p^{\infty}$-Selmer group of $A$\\
$T_\ell(A)$& the $\ell$-adic Tate module of $A$\\
$V_\ell(A)$& $T_\ell(A)\otimes_{\mathbb{Z}_\ell}\mathbb{Q}_\ell$\\
$\Omega^{1}(A)$ & the vector space of regular differentials on $A$\\
$\mu(X)$ & encodes whether $X$ is deficient, see Definition \ref{def:deficiency}\\
$\langle\cdot,\cdot\rangle$ & the inner product on characters of $G$-representations \\
$\rho^H$ & for a $G$-representation $\rho$, and subgroup $H\leq G$, denotes the subspace of elements of $\rho$ fixed by $H$.
\end{tabular}\medskip

We will usually write $K$ to denote a number field. In addition, we sometimes write $X/K$ to indicate that the curve $X$ is defined over $K$. While there is potential for confusion with the notation for quotient curves, we hope the meaning will be clear from context.

\subsection{Acknowledgements}
We want to thank Vladimir Dokchitser and Holly Green for many insightful discussions and valuable suggestions. The second author is supported by the Engineering and Physical Sciences Research Council (EPSRC) grant EP/V006541/1 ‘Selmer groups, Arithmetic Statistics and Parity Conjectures'. During part of the period in which this work was carried out, they were supported by the Leibniz fellow programme at the Mathematisches Forschungsinstitut Oberwolfach, and would like to express their thanks for the excellent working conditions provided.

\section{\'{E}tale algebras and curves} \label{s:appAdam}

In the following two sections, we prove Proposition \ref{Prop:Curves-Algebras} and define the notion of the $S_n$-closure of a cover of curves. We begin with a brief review of \'{e}tale algebras and their corresponding Galois sets. We then discuss curves and their Jacobians. Recall that, as in Convention \ref{convention}, we assume that all curves are smooth and proper, but do not assume that they are connected, nor that their connected components are geometrically connected. As a result, while much of the material in this section is presumably standard, we have often been unable to locate suitable references.  For completeness, we show how to deduce all relevant properties of these curves and their Jacobians from the corresponding, and more standard, results in the geometrically connected case. 

\subsection{\'{E}tale algebras}

Let $F$ be a field and let $F^s$ be a fixed separable closure of $F$. Denote by $G_F=\textup{Gal}(F^s/F)$ the absolute Galois group of $F$. Given a  finite  \'{e}tale algebra $L$, we write 
\[S(L)=\textup{Hom}_F(L,F^s)\]
for the finite $G_F$-set of $F$-algebra homomorphisms $L\rightarrow F^s$. Here, the action of $\sigma \in G_F$ on $\phi \in \textup{Hom}_F(L,F^s)$ is given  
$(\sigma \cdot \phi)(x)=\sigma \phi(x)$, for $x\in L$.
If $L/F$ has degree $n$, then $S(L)$ has cardinality $n$.

\begin{remark} \label{rem:roots_of_poly}
Suppose $L=F[x]/(f(x))$, where $f(x)\in F[x]$ is a separable polynomial. Then $S(L)$ is isomorphic to the $G_F$-set of roots of $f(x)$ in $F^s$. 
\end{remark}

Starting instead with a finite $G_F$-set $S$, we write 
\[L(S)=\textup{Map}_{G_F}(S,F^s)\]
for the $F$-algebra of $G_F$-equivariant maps $S\rightarrow F^s$.  

\begin{proposition} \cite[Chapter V.10.10]{MR1080964}\label{prop:etale_sets}
The assignments $L\mapsto S(L)$ and $S\mapsto L(S)$ induce a contravariant equivalence  of categories
\[\left\{\begin{array}{c}\textup{finite \'{e}tale }$F$\textup{-algebras} \\\textup{ and }\\ $F$\textup{-algebra homomorphisms} \end{array}\right\}  \stackrel{ }{\longleftrightarrow} \left\{\begin{array}{c}\textup{finite discrete }G_F\textup{-sets} \\\textup{ and } \\ G_F\textup{-equivariant maps}\end{array}\right\} .\]
\end{proposition}

\subsubsection{$S_n$-closure} \label{Snclo_subsec}

Let $F$, $L$ and $S(L)$ be as above,  let $n$ be the degree of $L$ over $F$, and write $[n]=\{1,...,n\}$.  Denote by $\textup{Bij}([n],S(L))$ the $G_F$-set of bijections from $[n]$ to $S(L)$. Here the action of $G_F$ on $\phi \in \textup{Bij}([n],S(L))$ is given by 
$(\sigma \cdot \phi)(m)=\sigma \phi(m)$, for $m\in [n].$

\begin{definition} \label{def: Sn-closure}
The \emph{$S_n$-closure} $\widetilde{L}$ of $L/F$ is defined to be the   \'{e}tale $F$-algebra associated via Proposition \ref{prop:etale_sets} to $\textup{Bij}([n],S(L))$.
\end{definition}

\begin{remark}
See \cite{MR3273311} for an alternative description of the $S_n$-closure which carries over to more general ring extensions. 
\end{remark}

Note that  $\textup{Bij}([n],S(L))$ carries a natural right action of the symmetric group $S_n$, given by precomposition. This  commutes with the $G_F$-action, so by  Proposition \ref{prop:etale_sets} we obtain a left action of $S_n$ on $\widetilde{L}$ by $F$-algebra automorphisms. Evaluation at $1\in [n]$ defines a $G_F$-equivariant surjection 
\begin{equation} \label{eq:surj_G_set}
\textup{Bij}([n],S(L)) \twoheadrightarrow S(L).
\end{equation}
Via Proposition \ref{prop:etale_sets} this induces an inclusion $L\hookrightarrow \widetilde{L}$ of $F$-algebras. In particular, $\widetilde{L}$ is naturally an extension of $L$.

\begin{remark} \label{rem_s_n_clo}
The map \eqref{eq:surj_G_set} realises $S(L)$ as the quotient of $\textup{Bij}([n],S(L))$ by the action of $1\times S_{n-1}$. From this we deduce that $L$ is the algebra of fixed points for the action of $1\times S_{n-1}$ on $\widetilde{L}$, i.e. $L=\widetilde{L} ^{1\times S_{n-1}}.$
\end{remark}

\begin{lemma} \label{lem: Discriminant Curve}
Suppose that the characteristic of $F$ is different from $2$. Further, suppose that $L=F[x]/(f(x))$ for some separable polynomial $f(x)\in F[x]$. Write $\Delta \in F^\times$ for the polynomial discriminant of $f(x)$. 

Then the algebra $\widetilde{L}^{A_n}$ of fixed points for the action of the alternating group $A_n$ is isomorphic, as an $F$-algebra,  to 
$F[y]/(y^2-\Delta).$
\end{lemma}

\begin{proof}
 Let $\mathcal{R}=\{x_1,...,x_n\}$ be the $G_F$-set of roots of $f(x)$ in $F^s$. From Remark \ref{rem:roots_of_poly} and the definition of $\widetilde{L}$, we see that $\widetilde{L}^{A_n}$ corresponds under Proposition \ref{prop:etale_sets} to the quotient of $\textup{Bij}([n],\mathcal{R})$ by the action of $A_n$. On the other hand, from the definition of the discriminant, we see that $F[y]/(y^2-\Delta)$ corresponds to the two element $G_F$-set
 \[S=\Big\{\pm \prod_{i<j}(x_i-x_j)\Big\}.\]
Since these two $G_F$-sets are isomorphic, the result follows from Proposition \ref{prop:etale_sets}.
\end{proof}

\subsubsection{G-closure} With a view to applications, we discuss a related notion to that of $S_n$-closure that works for more general groups. Let $L$ be an \'{e}tale algebra and let $\textup{Sym } S(L) $ denote the group of permutations of $S(L)$. We make $\textup{Sym } S(L) $ into a $G_F$-set by setting 
\[\qquad \qquad\qquad\qquad\qquad(\sigma\cdot \phi)(x)= \sigma (\phi(x)),\quad \quad \sigma \in G_F,~ \phi \in \textup{Sym }S(L), ~x\in S(L).\] 
Note that a choice of ordering on $S(L)$ identifies the \'{e}tale algebra associated to $\textup{Sym } S(L)$ with the $S_n$-closure of $L$. 

\begin{definition} \label{defi_G_closure}
Let $G$ be a subgroup of $\textup{Sym } S(L)$. Suppose that $G$ acts transitively on $S(L)$, and contains the image of the natural map $G_F\rightarrow \textup{Sym } S(L)$ arising from the $G_F$-action on $S(L)$. The  action of $G_F$ on $\textup{Sym }S(L)$ described  above restricts to an action of $G_F$ on $G$. We define the  \textit{$G$-closure} of $L$, denoted $L(G)$, to be the \'{e}tale $F$-algebra associated to $G$.
\end{definition}

\begin{remark}
Under the assumption that $G$ contains the image of $G_F\rightarrow \textup{Sym } S(L)$, the assumption that $G$ acts  transitively on $S(L)$ is automatic if $L$ is a field, since then $G_F$ acts transitively on $S(L)$. In this case, if one takes $G$ to be the image of $G_F$ in $\textup{Sym } S(L)$, then $L(G)$ is the Galois closure of the field extension $L/F$.
\end{remark}

Suppose that $G$ satisfies the assumptions of Definition \ref{defi_G_closure}. The action of $G$ on itself by right multiplication commutes with the $G_F$-action, so we obtain a left $G$-action on the $G$-closure of $L$ by $F$-algebra automorphisms.  Choose $x\in S(L)$ and let $G_x$ denote the stabiliser of $x$. Then the map $\phi\mapsto \phi(x)$ gives a surjection of $G_F$-sets $\pi_x:G\twoheadrightarrow S(L)$, realising $S(L)$ as the quotient of $G$ by $G_x$. In particular, this induces an $F$-algebra inclusion \[\iota_x:L\hookrightarrow L(G),\] which identifies $L$ with the  $F$-subalgebra of $L(G)$ consisting of elements fixed by $G_x$.

\begin{remark}
The inclusion of $L$ into $L(G)$ is independent of the choice of $x$ in the following sense. Given $x,y\in X$, say with $y=\phi(x)$ for some $\phi\in G$, we have $\pi_x=\pi_y R_{\phi^{-1}}$, where $R_{\phi^{-1}}:G\rightarrow G$ denotes right multiplication by $\phi^{-1}$. In particular, we have a commutative triangle of $F$-algebras 
\[
\xymatrix{L\ar[dr]^{\iota_y}\ar[r]^{\iota_x}&L(G) \ar[d]^{\phi} \\&L(G).}
\] 
\end{remark}

\begin{remark}
If $G$ contains the image of $G_F\rightarrow \textup{Sym } S(L)$, but the action of $G$ on $S(L)$ is not assumed to be transitive, then Definition \ref{defi_G_closure} still makes sense and produces an $F$-algebra $L(G)$. However, in this case there need not be any $F$-algebra inclusion $L\hookrightarrow L(G)$ (as is the case, for example, when $L=F\times F$ and $G$ is trivial). Because of this, we have elected not to define the $G$-closure in this generality.
\end{remark}

\subsection{Curves and function fields} \label{sec:curves_and_function_fields}

In what follows, we take $k$ to be a field of characteristic $0$. Let $X/k$ be a connected curve. Then $X$ is integral, so we may consider the function field $k(X)$ of $X$. One has that $X$ is geometrically connected if and only if $k(X)\cap \bar{k}=k$, in which case the function field of $X_{\bar{k}}$ is equal to $k(X)\otimes_k \bar{k}$ \cite[Corollary 3.2.14]{MR1917232}.
 
Given a non-constant morphism of connected $k$-curves $X\rightarrow Y$, there is an induced $k$-algebra homomorphism $k(Y)\hookrightarrow k(X)$. 
 
 \begin{proposition} \cite[ Tag 0BY1]{stacks-project} \label{curves_vs_function_fields}
 The map sending a connected $k$-curve to its function field induces a contravariant equivalence of categories between:
 \begin{itemize}
 \item[(1)] connected $k$-curves and non-constant $k$-morphisms,  
 \end{itemize}
 \begin{itemize}
 \item[(2)] finitely generated field extensions $F/k$ of transcendence degree $1$ and $k$-algebra homomorphisms.
 \end{itemize}
 \end{proposition}
 
 Henceforth, we refer to a finitely generated field extension  $F/k$ of transcendence degree $1$ simply as a `function field'.
 
\begin{remark} \label{nice_vs_nice_and_geom_irred}
 Suppose that $F/k$ is a function field, and write $k_F$ for the maximal algebraic extension of $k$ contained in $F$. Then $k_F$ is a finite extension, and $F/k_F$ is a function field in which $k_F$ is algebraically closed. Thus $F$ corresponds via Proposition \ref{curves_vs_function_fields} both to  a geometrically connected curve $X$ over $k_F$, and a connected curve over $k$. The latter is the scheme $X$ equipped with the structure map  $X\rightarrow \textup{Spec }k_F \rightarrow \textup{Spec }k.$
\end{remark}

\subsection{General curves} \label{sec:quasi_nice_curve_defs}

We now turn to the more general curves satisfying Convention \ref{convention}. Such a curve $X$ is a finite disjoint union of connected $k$-curves. One advantage of considering such curves is that  their base-change to an arbitrary field extension $k'/k$ again satisfies Convention~\ref{convention}.

 \begin{notation}  \label{function_algebra_defi}
 Let $X/k$ be a curve. We denote by $\mathcal{K}_X$ the sheaf of total quotient rings on $X$ in the sense of  \cite[ Tag 02AR]{stacks-project}, and write $k(X)=\Gamma(X,\mathcal{K}_X)$ for the $k$-algebra of global sections of $\mathcal{K}_X$. 
 \end{notation}
 
Writing  $X=\bigsqcup_i X_i$ as a disjoint union of connected $k$-curves, we have $k(X)=\prod_i k(X_i)$. In particular, we can write this decomposition as
 \begin{equation} \label{decomp_of_quasi_nice}
 X=\bigsqcup_{x\in \textup{ Spec } k(X)}X_x,\end{equation}
 where, for $x\in \textup{ Spec }k(X)$, we denote by $X_x$ the connected $k$-curve corresponding via Proposition~\ref{curves_vs_function_fields} to the residue field of $x$. 

 \subsection{Covers of curves and \'{e}tale algebras}

Let $Y/k$ be a geometrically connected $k$-curve, and let $F=k(Y)$ be the function field of $Y$.  We will be interested in the following notion of cover of $Y$. 

\begin{definition} \label{def:cover_of_curve}
A \textit{cover} $f:X\rightarrow Y$ consists of:
\begin{itemize}
\item a  $k$-curve $X$ and,
\item a finite surjective $k$-morphism $f:X\rightarrow Y$, via which we view $X$ as an $Y$-scheme. 
\end{itemize}
A \textit{morphism of covers} $X\rightarrow X'$ is a finite surjective $Y$-morphism from $X$ to $X'$. 
We denote by $\textbf{Cov}_{Y}$ the category of covers of $Y$.
\end{definition}

We will show that $\textup{\textbf{Cov}}_{Y}$ is contravariantly equivalent to the category $\textbf{\'{E}t}_F^{\textup{inj}}$ consisting of \'{e}tale $F$-algebras along with injective $F$-algebra homomorphisms.
 
\begin{notation}
Let $f:X\rightarrow Y$ be a cover, and let $k(X)$ be as in Notation \ref{function_algebra_defi}. The morphism $f:X\rightarrow Y$ induces an inclusion $F\hookrightarrow k(X)$, making $k(X)$ into an \'{e}tale $F$-algebra. Similarly, given a morphism of covers $\phi:X\rightarrow Y$, pullback of functions defines an injective $F$-algebra homomorphism $\phi^*:k(Y)\hookrightarrow k(X)$. 
\end{notation}

  \begin{proposition} \label{curves_vs_etale_algebras} 
 The map sending a cover $X\rightarrow Y$ to the $F$-algebra $k(X)$   induces a contravariant equivalence of categories between $\textup{\textbf{Cov}}_{Y}$ and $\textbf{\textup{\'{E}t}}_F^{\textup{inj}}$.
 \end{proposition}
 
 \begin{proof}
 This is a consequence of  Proposition \ref{curves_vs_function_fields}. For convenience, we describe how to construct an inverse functor.  
 
Given an \'{e}tale $F$-algebra $L$, the product of the natural quotient maps gives an isomorphism of $F$-algebras
\[L\stackrel{\sim}{\longrightarrow}\prod_{x \in \textup{Spec } L}L(x).\]
Each residue field $L(x)$ is a finite field extension of $F$, so corresponds via Proposition \ref{curves_vs_function_fields} to a connected $k$-curve $X_x$ equipped with a finite surjective $k$-morphism $f_x:X_x\longrightarrow Y$.
Define the cover $X^L$ associated to $L$ to be  the disjoint union over $x\in \textup{Spec }L$ of the curves $X_x$, equipped with the finite surjective $k$-morphism   
\[\bigsqcup_{x \in \textup{Spec } L} f_x:X^L\rightarrow Y.\]

Now suppose we have an injective homomorphism of \'{e}tale $F$-algebras $\phi:L\rightarrow L'$, and write $\phi'$ for the induced map $\textup{Spec }L'\rightarrow \textup{Spec }L$. To ease notation, take $X=X^L$ and $X'=X^{L'}$ to be the corresponding curves. For each $x\in \textup{Spec } L'$, the homomorphism $L(\phi'(x))\rightarrow L'(x)$ induced by $\phi$ corresponds, via Proposition \ref{curves_vs_function_fields}, to a finite surjective $Y$-morphism 
\[\phi_x:X'_x\longrightarrow X_{\phi'(x)}.\]
Together, the $\phi_x$ give a finite $Y$-morphism $X'\rightarrow X$. To show that this map is surjective we first note that, since $\phi$ is injective, $\phi'$ is surjective. Thus every connected component of $X$ appears as $X_{\phi'(x)}$ for some $x\in \textup{Spec }L'$. This and surjectivity of the individual $\phi_x$ gives the result.
 \end{proof}

\begin{remark}
 Let $k'/k$ be a field extension. Given a cover $X\rightarrow Y$, its base-change to $k'$ \hbox{$X\times_{k}k'\rightarrow Y\times_{k}k'$} is a cover of the geometrically connected curve $k'$-curve $Y\times_{k}k'$. Since also $k'(X\times_{k}k')=k(X)\otimes_k k'$,  we see that
 the correspondence of Proposition \ref{curves_vs_etale_algebras} is compatible with base-change to arbitrary field extensions.
 \end{remark}
 
 \subsection{The $S_n$-closure of a cover of curves}
Proposition \ref{curves_vs_etale_algebras} allows us to  define the $S_n$-closure of any cover $X\rightarrow Y$, where $n$ is the degree of $k(X)$ over $k(Y)$. 
 
\begin{definition}\label{def:snclosurecurve} 
Let $f:X\rightarrow Y$ be a cover,  let $L=k(X)$ the   corresponding \'{e}tale $F$-algebra, and let $n=[L:k(Y)]$. Then we define the \emph{$S_n$-closure} of $X\rightarrow Y$ to be the cover   $\widetilde{X}\rightarrow Y$ corresponding, via  Proposition \ref{curves_vs_etale_algebras}, to the $S_n$-closure of $L$ (cf. Definition \ref{def: Sn-closure}).
\end{definition}
 
\begin{remark}
More generally, via Definition \ref{defi_G_closure}, one can define the $G$-closure of $X\rightarrow Y$ for suitable subgroups $G$ of $S_n$.
\end{remark}

\section{Jacobians}\label{sec:jacobians_quasi_nice_curves}
We now shift our attention from curves to Jacobians, with Convention \ref{convention} remaining in place. The purpose of this section is to explain how several standard results for Jacobians of geometrically connected curves extend to the slightly more general setting of Convention \ref{convention}.  As before, $k$ denotes a field of characteristic $0$. 

\subsection{The geometrically connected case} \label{sec:curves_and_jacobians_background}
 Let $X$ be a geometrically connected $k$-curve. We denote by $\Jac_{X}$ the Jacobian of $X$, which is an abelian variety over $k$ of dimension equal to the genus of $X$. Denote by $\lambda_X$ the canonical principal polarisation on $\Jac_{X}$.
Given a non-constant morphism $f:X\rightarrow Y$ of geometrically connected $k$-curves, we have  induced $k$-homomorphisms 
\[f^*:\Jac_{Y}\rightarrow \Jac_{X}\quad \textup{ and }\quad f_*:\Jac_{X} \rightarrow \Jac_{Y}.\]

We have $\Jac_{X}(\bar{k})=\textup{Pic}^0(X_{\bar{k}})$, which identifies with the group of degree $0$ divisors on $X_{\bar{k}}$ modulo linear equivalence, and similarly for $Y$. Under this identification,   $f_*$ sends a prime divisor $Q$ to $f(Q)$, whilst $f^*$ is the homomorphism that sends a prime divisor $Q$ to
\[ \sum_{P\in f^{-1}(Q)}e(P)  P,\]
where $e(P)$ is the ramification degree of $f$ at $P$.

The following lemma is standard, but we have been unable to locate a suitable reference. 

\begin{lemma} \label{lem:fstar_duality}
We have  $f^*=\lambda_X^{-1}\circ f_*^\vee \circ \lambda_Y$, where $f_*^\vee$ is the dual of $f_*$.
\end{lemma}

\begin{proof}
It suffices to check this identity after base-change to $\bar{k}$. 
Fix any point $P$ in $X(\bar{k})$. Then $P$ defines a morphism (the Abel--Jacobi map) $\phi^P:X\rightarrow \Jac_{X}$ induced by the map on divisors 
$(Q) \longmapsto (Q)-(P) .$
Similarly, the  point $f(P) \in Y(\bar{k})$ induces a morphism $\phi^{f(P)}:Y\rightarrow \Jac_{Y}$. From the explicit description of $f_*$ in terms of divisors, one checks that  
\begin{equation} \label{ab_jac_com}
\phi^{f(P)}\circ f=f_*\circ \phi^P.
\end{equation}
 After applying the functor $\textup{Pic}^0_{-/k}$ to this equality, the result  follows from  \cite[Lemma 6.9]{MR861976}.
\end{proof}

\begin{remark}
Using the principal polarisations $\lambda_X$ and $\lambda_Y$ to identify the Jacobians of $X$ and~$Y$ with their duals, we can express Lemma \ref{lem:fstar_duality} more succinctly as saying that $(f_*)^\vee=f^*$.
\end{remark}

\subsection{The general case} \label{sec:quasi_nice_jac}
Let $X/k$ be a curve satisfying Convention \ref{convention}. By definition, we take the Jacobian $\Jac_{X}$ of  $X$ to be the identity component of the relative Picard functor of $X/k$. Let $k(X)$ be as in Definition \ref{function_algebra_defi}.    In what follows, we write $\overline{X}=X_{\bar{k}}$ for the base-change of $X$ to $\bar{k}$, and write $\bar{k}(X)$ as shorthand for $k(X)\otimes_k\bar{k}=\bar{k}(\overline{X})$. We have the following basic lemma.
 
 \begin{lemma} \label{is_an_ab_var}
 The Jacobian of $X$ is a principally polarised abelian variety over $k$. 
 \end{lemma}
 
 \begin{proof}
Since $X$ is proper, by  \cite[Theorem 2]{MR206011} and \cite[Lemma 5.1]{MR2223410} (cf. also \cite[Corollary 4.18.3]{MR2223410}),  $\textup{Pic}^0_{X/k}$ is represented by a finite-type commutative group scheme over $k$. Now decompose $\overline{X}$ as in \eqref{decomp_of_quasi_nice}, noting that, as we are now working over an algebraically closed field, each $\overline{X}_x$ is a geometrically connected curve.  Then
\begin{equation} \label{jac_kbar}
(\Jac_{X})_{\bar{k}}\stackrel{\textup{\cite[Prop. 5.3]{MR2223410}} }{=}\textup{Pic}^0_{\overline{X}/\bar{k}}=\prod_{x\in \textup{ Spec }\bar{k}(X)}\textup{Jac}_{\overline{X}_x} 
\end{equation}
 is an abelian variety. Thus the same is true of $\Jac_{X}$.  
 
With the curves $\overline{X}_x$ as above, equip each $\textup{Jac}_{\overline{X}_x} $ with its canonical principal polarisation, and equip  $(\Jac_{X})_{\bar{k}}$ with the corresponding product polarisation, which we denote $\lambda_X$. We wish to show that $\lambda_X$ is defined over $k$. To see this, let $P=(P_x)_{x\in \textup{ Spec } \bar{k}(X)}$ be a tuple of $\bar{k}$-points, with $P_x \in \overline{X}_x(\bar{k})$ for each $x$. Let
 \[\phi^{(P)}:\overline{X}=\bigsqcup_{x\in \textup{ Spec } \bar{k}(X)}\overline{X}_x\longrightarrow \prod_{x\in \textup{Spec }\bar{k}(X)}\textup{Jac}_{\overline{X}_x}= (\Jac_{X})_{\bar{k}} \]
 denote the morphism induced by the individual Abel--Jacobi maps associated to the $P_x$. Pulling back line bundles under $\phi^{(P)}$  gives a homomorphism
$\textup{Jac}_X^\vee(\bar{k})\longrightarrow \textup{Jac}_X(\bar{k}).$
  From  \cite[Lemma 6.9]{MR861976} we see that this agrees with the map on $\bar{k}$-points induced by  $-\lambda_X^{-1}$. In particular, this map is independent of the choice of $P$. Applying this observation to the corresponding maps constructed from the Galois conjugates of $P$, we  conclude that $-\lambda_X^{-1}$, hence also $\lambda_X$,  is defined over $k$.  
 \end{proof}
 
 \subsubsection{Jacobians as products of Weil restrictions} We  record the following description of the Jacobian as a product of Weil restrictions of Jacobians of geometrically connected curves. We remark that this gives an alternative approach to Lemma \ref{is_an_ab_var}.
 
 \begin{notation}
 Let $k'/k$ be a finite field extension, and let $V$ be a quasi-projective $k'$-variety. We denote by $\textup{Res}_{k'/k}V$ the   \textit{Weil restriction}  of $V$ from $k'$ to $k$. By   \cite[Section 7.6]{MR1045822} it is a  $k$-variety which represents the functor from $k$-schemes to sets given by 
 $T \mapsto V(T\times_k k').$ Similarly,  for a functor $F$ from $k'$-schemes to sets, we denote by $\textup{Res}_{k'/k}F$ the functor  from $k$-schemes to sets given by 
 $\textup{Res}_{k'/k}F(T)=  F(T\times_k k').$
 \end{notation}
 
 As in \eqref{decomp_of_quasi_nice}, write   
 \[
X=\bigsqcup_{x\in \textup{ Spec } k(X)}X_x \]
 as a disjoint union of connected $k$-curves. For $x\in \textup{ Spec } k(X)$, write $k(X_x)$ for the function field of $X_x$, and write $k_x=\bar{k}\cap k(X_x)$. Per Remark \ref{nice_vs_nice_and_geom_irred}, $k(X_x)$ corresponds to a geometrically connected $k_x$-curve $X_x'$, and $X_x$ is obtained from $X_x'$ by forgetting the $k_x$-structure.
 
 \begin{lemma} \label{lem:curve_Weil restriction}
As abelian varieties over $k$, we have
 \[\textup{Jac}_X\cong \prod_{x\in  \textup{ Spec } k(X)}\textup{Res}_{k_x/k}\textup{Jac}_{X'_x}.\] 
 \end{lemma}
 
 \begin{proof}
The decomposition \eqref{decomp_of_quasi_nice} gives  $\textup{Jac}_X\cong \prod_{x\in \textup{ Spec } k(X)}\textup{Jac}_{X_x}$. Fixing $x\in \textup{Spec }k(X)$, we wish to show that $\textup{Jac}_{X_x}\cong \textup{Res}_{k_x/k}\textup{Jac}_{X_x'}$. Since $\textup{Res}_{k_x/k}\textup{Jac}_{X'_x}$ is connected, it  suffices to show that 
$\textup{Res}_{k_x/k}\textup{Pic}_{X_x'/k_x}\cong \textup{Pic}_{X_x/k}.$
Denote by $g:\textup{Spec }k_x\rightarrow \textup{Spec }k$ the morphism corresponding to the inclusion of $k$ into $k_x$, and denote by $f:X_x'\rightarrow \textup{Spec }k_x$ the structure morphism. Then as fppf sheaves we have 
\[\textup{Res}_{k_x/k}\textup{Pic}_{X_x'/k_x}=g_*R^1 f_* \mathbb{G}_{m}\quad\textup{ and }\quad \textup{Pic}_{X_x/k}=R^1 (g\circ f)_*\mathbb{G}_{m },\]
where $\mathbb{G}_m$ denotes the multiplicative group over $X_x'$. As $k_x/k$ is a finite separable extension, $g_*$ is exact on fppf sheaves (exactness can be checked after pullback to $k_x$, and $\textup{Spec }k_x~\times_{\textup{Spec }k}~\textup{Spec }k_x$ is a finite disjoint union of copies of $\textup{Spec }k_x$). Thus the functors $R^1(g_*f_*)$ and $g_*R^1(f_*)$ agree, giving the result.
 \end{proof}
 
  \subsubsection{Induced maps between Jacobians} \label{induced_map_jac_q_n}
 It what follows, given a $k$-curve $X$, we will always equip $\textup{Jac}_X$ with the principal polarisation $\lambda_X$ constructed in the proof of Lemma \ref{is_an_ab_var}. 

 \begin{definition} \label{quasi_nice_f_star}
 Let $X$ and $Y$ be curves over $k$, and let $f:X\rightarrow Y$ be a finite surjective morphism. We denote by $f^*:\textup{Jac}_Y\rightarrow \textup{Jac}_X$ the map induced by pullback of line bundles. For the purposes of this paper, we then \textit{define} $f_*:\textup{Jac}_X\rightarrow \textup{Jac}_Y$ to be the dual of $f^*$, i.e. we set 
 \[f_*=\lambda_Y^{-1}\circ (f^*)^\vee\circ \lambda_X.\]
 By Lemma \ref{lem:fstar_duality} this agrees with the usual notion for geometrically connected curves. 
 \end{definition} 

We will want to work explicitly with the maps  induced on $\bar{k}$-points by $f^*$ and $f_*$. For this we introduce the following notation.
 
 \begin{notation} \label{notat:points_on_Jacobians}
Decompose $\overline{X}=X\times_k\bar{k}$ as in  \eqref{decomp_of_quasi_nice}. Set $\textup{Div}(\overline{X})=\oplus_{P\in X(\bar{k})}\mathbb{Z}\cdot P,$ and denote by $\textup{Div}^0(\overline{X})$ the subgroup of $\textup{Div}(\overline{X})$ consisting of elements whose restriction to each connected component of $\overline{X}$ has degree $0$. Given $\eta\in \bar{k}(X)^\times$, we can view $\eta$ as a tuple $(\eta_x)_{x\in \textup{ Spec }\bar{k}(X)}$, where each $\eta_x$ is in $\bar{k}(\overline{X}_x)^\times$.  We then define 
 \[\textup{div}(\eta)=\sum_{x\in \textup{ Spec }\bar{k}(X)}\textup{div}(\eta_x)\in \textup{Div}^0(\overline{X}). \]
 (Equivalently, viewing $\eta$ as an element  of $\Gamma(\overline{X},\mathcal{K}_{\overline{X}})$, $\textup{div}(\eta)$ corresponds to the principal Cartier divisor associated to $\eta$.)
Writing $\textup{Prin}(\overline{X})$ for the subgroup of $\textup{Div}^0(\overline{X})$ of elements of the form $\textup{div}(\eta)$ for $\eta \in\bar{k}(X)^\times$, we see from \eqref{jac_kbar} that we have a canonical identification
\begin{equation*} \label{eq:jac_k_bar_points}
\textup{Jac}_X(\bar{k})=\textup{Div}^0(\overline{X})/\textup{Prin}(\overline{X}).
\end{equation*}
Given $Z\in \textup{Div}^0(\overline{X})$, we write $[Z]$ for its class in $\textup{Jac}_X(\bar{k})$. 
 \end{notation}
 
Now let $f:X\rightarrow Y$ be a finite surjective morphism.  As in the case of geometrically connected curves, we have homomorphisms
\begin{equation} \label{f_star_in_quasi_divs}
f^*:\textup{Div}(\overline{Y})\longrightarrow \textup{Div}(\overline{X})\quad \textup{ and }\quad f_*:\textup{Div}(\overline{X})\longrightarrow \textup{Div}(\overline{Y})
\end{equation}
determined by setting $f_*(P)=f(P)$ for all $P\in X(\bar{k})$, and setting
\[f^*(Q)=\sum_{P\in f^{-1}(Q)}e(P)P \] for all $Q\in Y(\bar{k})$. (Again, $e(P)$ denotes the ramification index of $f$ at $P$, i.e. the ramification index of $\mathcal{O}_{\overline{Y},f(P)}$ over $\mathcal{O}_{\overline{X},P}$.)

 \begin{lemma} \label{agree_with_obvious_maps_f_star}
 For all $Z\in \textup{Div}^0(\overline{X})$ and $Z'\in \textup{Div}^0(\overline{Y})$, we have 
 \[f^*[Z']=[f^*Z']\quad  \textup{ and }\quad f_*[Z]=[f_*Z].\]
 \end{lemma}
 
\begin{proof}
We can formally reduce to the case where $X$ and $Y$ are geometrically connected as follows. Write $\overline{X}$ and $\overline{Y}$ as disjoint unions of connected $\bar{k}$-curves as in  \eqref{decomp_of_quasi_nice}. 
For $x\in \textup{ Spec }\bar{k}(X)$, write $f_x$ for the finite surjective morphism $f_x:\overline{X}_x\rightarrow \overline{Y}_{f(x)}$ induced by $f$. Now (the base-change to $\bar{k}$  of) $f^*:\textup{Jac}_Y\rightarrow \textup{Jac}_X$ gives a homomorphism
\[\prod_{y\in \textup{ Spec }\bar{k}(Y)}\textup{Jac}_{\overline{Y}_y}\longrightarrow \prod_{x\in \textup{ Spec }\bar{k}(X)}\textup{Jac}_{\overline{X}_x}.\]
From the definition of $f^*$ we see that this corresponds to the matrix of homomorphisms $(\phi_{y,x})$, where for $y\in \textup{ Spec }\bar{k}(Y)$ and $x \in \textup{ Spec }\bar{k}(X)$, we have 
\[\phi_{y,x}=\begin{cases}f_x^*~~&~~f(x)=y,\\ 0~~&~~\textup{otherwise.} \end{cases}\]
Using this, the claim for $f^*$ reduces to the case of maps between geometrically connected curves. Dualising with respect to the product polarisation and applying Lemma \ref{lem:fstar_duality}, we see that   $f_*:\textup{Jac}_X\rightarrow \textup{Jac}_Y$ corresponds to the matrix $(\phi^\vee_{x,y})$ where, for $x \in \textup{ Spec }\bar{k}(X)$ and $y\in \textup{ Spec }\bar{k}(Y)$, we have
\[\phi^\vee_{x,y}=\begin{cases} (f_x)_*~~&~~f(x)=y,\\ 0~~&~~\textup{otherwise.}\end{cases} \]
The result for $f_*$ now similarly reduces to the geometrically connected case.  \end{proof}

\begin{remark} \label{g_and_dual_inverse_rmk}
If $g$ is a $k$-automorphism of $X$, then from the explicit description of $g^*$ and $g_*$ on $\bar{k}$-points afforded by Lemma \ref{agree_with_obvious_maps_f_star}, we see that $g^*=g_*^{-1}$.
\end{remark}

\begin{remark} \label{rmk:pull_back_div_commute_quasi_nice}
Given $f:X\rightarrow Y$ as above, denote by $f^*:\bar{k}(Y)\rightarrow \bar{k}(X)$ the induced map on global sections of sheaves of total quotient rings (over $\bar{k}$). Then given $\eta \in \bar{k}(Y)$, one has 
$f^*(\textup{div}(\eta))=\textup{div}(f^*\eta).$
To see this one can, for example, again reduce to the geometrically connected case. 
\end{remark}

\subsubsection{Comparison of differentials.}
Let $X$ be a $k$-curve, and write $\overline{X}=\bigsqcup_{x\in \textup{Spec }\bar{k}(X)}\overline{X}_x$ as in Section \ref{sec:quasi_nice_jac}. Recall from the proof of Proposition \ref{is_an_ab_var} that, for a tuple $P=(P_x)_x\in \prod_{x\in \textup{Spec }\bar{k}(X)}\overline{X}_x(\bar{k})$, we have a map $\phi^{(P)}:\overline{X}\rightarrow (\textup{Jac}_X)_{\bar{k}}$. This induces a homomorphism
\begin{equation} \label{differentials_pull_back_iso}
\phi^{(P)*}:\Omega^1((\textup{Jac}_X)_{\bar{k}})\longrightarrow \Omega^1(\overline{X}),
\end{equation}
where for a scheme $V$ over a field $F$, we denote by $\Omega^1(V)$ the global sections of the sheaf of differentials of $V$ over $F$.

\begin{proposition} \label{iso_of_diffs_vs_jac_prop}
The map \eqref{differentials_pull_back_iso} is an isomorphism independent of the choice of $P$. In particular, it is $\textup{Gal}(\bar{k}/k)$-equivariant and \eqref{differentials_pull_back_iso} descends to an isomorphism
\begin{equation} \label{diff_over_base_field}
\alpha_X:\Omega^1(\textup{Jac}_X)\stackrel{\sim}{\longrightarrow} \Omega^1(X).
\end{equation}
Given a finite surjective morphism $f:X\rightarrow Y$ between  $k$-curves, we have a commutative diagram 
\begin{equation*} \label{diag_com_diff}
\xymatrix{  \Omega^1(\textup{Jac}_Y)  \ar[d]_{\alpha_Y}\ar[r]^{(f_*)^*} &\Omega^1(\textup{Jac}_X)\ar[d]^{\alpha_X}  \\      
\Omega^1(Y)\ar[r]^{f^*}&\Omega^1(X).
}
\end{equation*}
\end{proposition}

\begin{proof}
For $X$ geometrically connected, the claim that $\eqref{differentials_pull_back_iso}$ is an isomorphism independent of $P$  is \cite[Proposition 2.2]{MR861976}. For the general case, recall that $\phi^{(P)}$ is defined as the morphism
\[\overline{X}=\bigsqcup_{x\in \textup{Spec }\bar{k}(X)}\overline{X}_x\longrightarrow \prod_{x\in  \textup{Spec }\bar{k}(X)}\textup{Jac}_{\overline{X}_x}=(\textup{Jac}_X)_{\bar{k}}\]
induced from the individual Abel-Jacobi maps $\phi^{(P_x)}$. We then have isomorphisms
\[\Omega^1((\textup{Jac}_X)_{\bar{k}})\cong \prod_{x\in  \textup{Spec }\bar{k}(X)}\Omega^1(\textup{Jac}_{\overline{X}_x})\cong\prod_{x\in  \textup{Spec }\bar{k}(X)}\Omega^1(\overline{X}_x) \cong \Omega^1(\overline{X}),\]
where the leftmost isomorphism is induced by pulling back along the projection maps $\textup{Jac}_X\rightarrow\textup{Jac}_{\overline{X}_x}$, and  the rightmost isomorphism is induced by pulling back along the inclusion maps $\overline{X}_x\rightarrow \overline{X}$. The central isomorphism is the product of the maps $\phi^{(P_x)*}$, which are isomorphisms independent of $P_x$ by the geometrically connected case. It follows from the definition of $\phi^{(P)}$ that the composite map from left to right is equal to $\phi^{(P)*}$, showing that $\phi^{(P)*}$ is an isomorphism independent of $P$. Having shown independence of $P$ it follows that $\phi^{(P)*}$ is Galois equivariant, hence induces the sought isomorphism \eqref{diff_over_base_field}. 

In the geometrically connected case, the claimed commutativity of the diagrams is a consequence of  \eqref{ab_jac_com}. The general case again formally reduces to this, using the description of $f_*$ on the Jacobians of the geometric components given in the proof of  Lemma \ref{agree_with_obvious_maps_f_star}. 
\end{proof}

\begin{remark}
For each prime $\ell$, an identical argument gives a canonical isomorphism \[H^1_{\acute{e}t}((\textup{Jac}_X)_{\bar{k}},\mathbb{Z}_\ell)\cong H^1_{\acute{e}t}(\overline{X},\mathbb{Z}_\ell)\] on \'{e}tale cohomology, making the analogous diagram commute. Here the input in the geometrically connected case is \cite[Corollary 9.6]{MR861976}. 
\end{remark}

\section{Finite group actions on  curves}

In this section, we consider the action of a finite group $G$ on curves and their Jacobians. We retain the notation from the previous section. In particular, $k$ denotes a field of characteristic $0$ and Convention \ref{convention} is in place.

\subsection{Quotients by finite groups}  \label{quasi_quotient}
Let $X$ be a $k$-curve, and let $G$ be a finite subgroup of $\textup{Aut}_k(X)$. Since $X$ is projective, the quotient scheme $X/G$ exists, and may be obtained by glueing the schemes $\textup{Spec } \mathcal{O}_X(U)^G$ as $U$ ranges over $G$-stable affine open subschemes of $X$. Recall the definition of $\textup{Div}^0(\overline{X})$ from Notation \ref{notat:points_on_Jacobians}.

\begin{proposition} \label{pull_back_pushforward_prop_quasi_nice}
The quotient $X/G$ is again a curve in the sense of Convention \ref{convention}. Moreover, writing $\pi:X\rightarrow X/G$ for the quotient morphism we have,  both as maps between Jacobians  and as maps between $\textup{Div}^0(\overline{X})$ and $\textup{Div}^0(\overline{X/G})$,
\[\pi_*\circ \pi^*=|G| \quad \textup{ and }\quad \pi^*\circ\pi_*=\sum_{g\in G}g_*.\]
\end{proposition}

\begin{proof}
When $X$ is geometrically connected this is standard. We formally reduce to this case as follows. First, base-changing to $\bar{k}$, it suffices to prove the result for $X$ replaced by $\overline{X}$, equipped with the base-changed action of $G$. Moreover, splitting into orbits for the action of $G$ on the connected components of $\overline{X}$, we reduce to the case where $G$ acts transitively on $\textup{Spec }\bar{k}(X)$. After these reductions, fixing $x\in \textup{ Spec }\bar{k}(X)$ and denoting by $G_x$ the stabiliser of $x$ in $G$, one finds that the quotient $\overline{X}/G$ is isomorphic to $Y=\overline{X}_x/G_x$. To describe the quotient map $\pi:\overline{X}\rightarrow Y$, for $\sigma \in G$  denote by $\overline{\sigma}^{-1}$ the isomorphism 
$ \overline{X}_{\sigma x}/G_{\sigma x}\cong Y$
induced by the action of $\sigma^{-1}$ on $\overline{X}$. Then  the restriction of $\pi$ to $\overline{X}_{\sigma x}$ is the composition
\[\overline{X}_{\sigma x}\longrightarrow \overline{X}_{\sigma x}/G_{\sigma x}\stackrel{\overline{\sigma}^{-1}}{\longrightarrow}Y,\] 
  the first morphism being the natural quotient map. This reduces the statement to the corresponding result for the actions of the groups $G_{\sigma x}$ on the geometrically connected curves  $\overline{X}_{\sigma x}$.
\end{proof}

\begin{remark} \label{rem:invariant_function_algebra}
Since $X/G$ is formed by glueing the schemes $\textup{Spec }\mathcal{O}_X(U)^G$ as $U$ ranges over $G$-stable  affine open subsets of $X$, we see that $k(X/G)=k(X)^G$. 
\end{remark}

\begin{remark} \label{rem:is_quotient}
Let $Y$ be a geometrically connected $k$-curve, let $F=k(Y)$ be its function field, and let $L/F$ be a finite \'{e}tale $F$-algebra.  Further, let $f:X\rightarrow Y$ be the cover corresponding to $L/F$ via Proposition \ref{curves_vs_etale_algebras}, and suppose that $G$ is a finite subgroup of $\textup{Aut}_F(L)$. Via Proposition~\ref{curves_vs_etale_algebras} we view $G$ as a finite subgroup of $\textup{Aut}_k(X)$. Let $X/G$ denote the quotient curve. From Remark~\ref{rem:invariant_function_algebra} we see that the morphism of covers  of $Y$ corresponding to the inclusion $L^G\hookrightarrow L$ is isomorphic to the quotient morphism $\pi:X\rightarrow X/G$. 
\end{remark}

\begin{remark}
Following on from the previous remark,   suppose we have a cover $X\rightarrow Y$ of curves, with $Y$ geometrically connected, and let $\widetilde{X}\rightarrow Y$ be its $S_n$-closure (cf. Definition \ref{def:snclosurecurve}). Then $Y$ is isomorphic to the quotient $\widetilde{X}/S_n$. Moreover, $\widetilde{X}\rightarrow Y$ factors through the initial cover $X\rightarrow Y$, and identifies $X$ with the quotient of $\widetilde{X}$ by $1\times S_{n-1}$. 
\end{remark}

\subsection{Equivariant Riemann--Hurwitz}
\label{subsec:equiv_riemann_hurwitz} 
Let $Y$ be a geometrically connected $k$-curve of genus $g(Y)$ and let $\pi:X \to Y$ be a $G$ cover. We now prove Proposition \ref{Thm:Equivariant Riemann--Hurwitz}. 
The case where $X$ is geometrically connected is the equivariant Riemann--Hurwitz formula  \cite[Proposition 1.1]{Riemann-Hurwitz}. The general case, whose proof reduces to this, is as follows. For the statement, note that the group $G$ acts transitively on the set of connected components of $\overline{X}=X_{\bar{k}}$ (since $Y$ is geometrically connected). For a connected component $\Gamma$ of $\overline{X}$,  we denote by $\textup{Stab}_G(\Gamma)$ its stabiliser in $G$, so that elements of $\textup{Stab}_G(\Gamma)$ act on $\Gamma$ as $\bar{k}$-automorphisms.

\begin{proposition} \label{Theorem:equiv_RH}
Let $\Gamma$ be a connected component of $\overline{X}$, write $H=\textup{Stab}_G(\Gamma)$, and write $\pi|_\Gamma: \Gamma \to Y$ for the restriction of $\pi$ to $\Gamma$. 
Denote by $\{q_1,\dots,q_r\} \subseteq Y(\bar{k})$  the branch points  of $\pi|_\Gamma$. For each $i$, let $t_i \in \pi|_\Gamma^{-1}(q_i)$ be any choice of preimage of $q_i$.   Then, for every prime $\ell$, the representations
\[V_{\ell}(\textup{Jac}_{X}) \ \ \textup{and} \ \ \textup{Ind}_{H}^{G} \mathds{1}^{\oplus 2} \oplus \textup{Ind}_{\{e\}}^{G}\mathds{1}^{\oplus(2\textup{g}(Y)-2) }\oplus \bigoplus_{i=1}^{r} (\textup{Ind}_{\{e\}}^{G} \mathds{1} \ominus \textup{Ind}_{\textup{Stab}_H(t_i)}^{G}\mathds{1})  \]
are isomorphic after extending scalars to $\mathbb{C}$.  
\end{proposition}

The following lemma will suffice for the reduction to the geometrically connected case. With a view to applications, we state it slightly more generality than needed. 

\begin{lemma} \label{orbs_of_comps_lemma}
Let $Z$ be a connected component of $X$ and let $H=\textup{Stab}_G(Z)$ denote the stabiliser of $Z$ in $G$. Then we have an isomorphism of $\mathbb{Q}_{\ell}[G]$-modules 
\begin{equation}
V_{\ell}(\textup{Jac}_X)\cong \textup{Ind}_H^GV_{\ell}(\textup{Jac}_Z).
\end{equation}
Similarly, we have an isomorphism of $k[G]$-modules
\begin{equation} \label{dif_ind_jac_omega}
\Omega^1(\textup{Jac}_X)\cong \textup{Ind}_H^G\Omega^1(\textup{Jac}_Z).
\end{equation}
\end{lemma}

\begin{proof}
As above, the assumption that $Y$ is (geometrically) connected ensures that $G$ acts transitively on the irreducible components of $X$. Fix a left transversal $y_1,...,y_n$ for $H$ in $G$, so that $y_1Z,...,y_nZ$ is the collection of connected components of $X$. Then $X=\bigsqcup_{i=1}^ny_iZ$ and we have an isomorphism of $k$-schemes
\[\alpha:\textup{Jac}_X  \stackrel{\sim}{\longrightarrow} \prod_{i=1}^n\textup{Jac}_{y_iZ} \]
induced by pullback along the inclusions  $y_iZ\hookrightarrow X$. This induces isomorphisms
\[V_{\ell}(\textup{Jac}_X)\cong \bigoplus_{i=1}^nV_{\ell}(\textup{Jac}_{y_iZ})\quad \textup{ and }\quad \Omega^1(\textup{Jac}_X)\cong \bigoplus_{i=1}^n\Omega^1(\textup{Jac}_{y_iZ}),\]
which give the isomorphisms of representations claimed in the statement once the $G$-action is taken into account.  
\end{proof}

\begin{remark}
By Proposition \ref{iso_of_diffs_vs_jac_prop}, the $k[G]$-modules  $\Omega^1(\textup{Jac}_X)$ and $\Omega^1(X)$ are isomorphic. Using this, and the corresponding statement for $Z$,  we could equivalently state \eqref{dif_ind_jac_omega} in terms of differentials on $X$.
\end{remark}

\begin{proof}[Proof of Proposition \ref{Theorem:equiv_RH}]
By considering the $G$-cover $\pi\times 1:\overline{X}\rightarrow \overline{Y}$, we see that we may assume $k$ to be algebraically closed.  Working under this assumption, and with $\Gamma$ and $H=\textup{Stab}_G(\Gamma)$ as in the statement, Lemma \ref{orbs_of_comps_lemma} gives an isomorphism of $\mathbb{Q}_\ell[G]$-representations 
\begin{equation}\label{induce_jac_vl}
V_{\ell}(\textup{Jac}_{X})\cong \textup{Ind}_{H}^{G} V_{\ell}(\textup{Jac}_{\Gamma}).
\end{equation} 
 To complete the proof we can appeal to the more standard form of equivariant Riemann--Hurwitz, in which $X$ is assumed geometrically connected. Specifically, $\pi|_\Gamma:\Gamma \rightarrow Y$ is an $H$-cover (as can be checked by considering the corresponding map on function fields). It  follows from \cite[Proposition 1.1]{Riemann-Hurwitz}  that, after extending scalars to $\mathbb{C}$, the $H$-representation $V_{\ell}(\textup{Jac}_{\Gamma})$ becomes isomorphic to 
 \[  \mathds{1}^{\oplus 2} \oplus \textup{Ind}_{\{e\}}^{H} \mathds{1}^{\oplus(2\textup{g}(Y) -2)} \oplus \bigoplus_{i=1}^{r} \big{(} \textup{Ind}_{\{e\}}^{H} \mathds{1} \ominus \textup{Ind}_{\textup{Stab}_H(t_i)}^{H} \mathds{1} \big{)}.\]
 The result now follows by inducing  from $H$ to $G$ and using \eqref{induce_jac_vl}.
\end{proof}

\subsection{From $G$-maps to isogenies}
\label{sec:Homomorphisms_between_jacobians}
As above, let $X$ be a curve over $k$, and let $G$ be a finite subgroup of  $\textup{Aut}_k(X)$. In the rest of this section, we show that one can naturally associate homomorphisms between Jacobians of quotients of $X$  to certain $G$-equivariant maps between permutation modules. This idea originates in work of Kani--Rosen \cite{MR1000113}, though we follow more closely subsequent work of  Chen  \cite{MR1779608}, and de Smit and Edixhoven \cite{MR1764312}. See Remark \ref{rem:desmit_chen_remark} below for a comparison between our main result here, Theorem \ref{main_G_mod_corresp_thm}, and those works. We will use the following notation concerning permutation modules.
\begin{notation}\label{notat:G_mod_homs_invars} 
Let $M$ be a $\mathbb{Z}[G]$-module. For each subgroup $H$ of $G$ we have an isomorphism 
\begin{equation}\label{eq:invariants_as_hom}
 \textup{Hom}_G(\mathbb{Z}[G/H],M)\stackrel{\sim}{\longrightarrow}M^{H}
\end{equation}
 given by evaluating homomorphisms at the trivial coset. Now suppose we have subgroups $H_1,...,H_n$ and $H_1',...,H_{m}'$   of $G$, and define $G$-sets 
\[S=\bigsqcup_{i=1}^nG/H_i\quad \textup{ and }\quad  S'=\bigsqcup_{j=1}^{m}G/H_j' .\] 
Taking $\mathbb{Z}[S]$ and $\mathbb{Z}[S']$ to be the corresponding permutation modules, \eqref{eq:invariants_as_hom} induces  isomorphisms 
\[\textup{Hom}_G(\mathbb{Z}[S],M)\cong \bigoplus_{i}M^{H_i}\quad \textup{ and }\quad  \textup{Hom}_G(\mathbb{Z}[S'],M)\cong \bigoplus_{j}M^{H_j'}.\]    
 Consequently, given $\phi$ in $\textup{Hom}_G(\mathbb{Z}[S],\mathbb{Z}[S'])$, the  map from $\textup{Hom}_G(\mathbb{Z}[S'],M)$ to $\textup{Hom}_G(\mathbb{Z}[S], M)$ sending $f$ to $f\circ \phi$ induces  a homomorphism
  \begin{equation} \label{eq:induced_G_hom}
 \phi^*  : \bigoplus_{j}M^{H_j'}\rightarrow \bigoplus_{i}M^{H_i}.
 \end{equation}
 Moreover, the $G$-module $\mathbb{Z}[S]$ (resp. $\mathbb{Z}[S']$) is canonically self-dual, via the pairing making the elements of $S$ (resp. $S'$) an orthonormal basis. Given $\phi$ in $\textup{Hom}_G(\mathbb{Z}[S],\mathbb{Z}[S'])$, we denote  by $\phi^\vee$ the corresponding dual homomorphism $\phi^\vee:\mathbb{Z}[S']\rightarrow \mathbb{Z}[S]$.
\end{notation}

\begin{theorem} \label{main_G_mod_corresp_thm} 
 Let $H_1,...,H_n$ and $H_1',...,H_m'$ be subgroups of $G$, take 
 \[S=\bigsqcup_{i } G/H_i\quad \textup{ and }\quad S'=\bigsqcup_{j}G/H_j',\]
 and let $\mathbb{Z}[S]$ and $\mathbb{Z}[S']$ denote the corresponding permutation modules. 

Then for each $G$-module homomorphism $\phi:\mathbb{Z}[S]\rightarrow \mathbb{Z}[S']$, there is a $k$-homomorphism of abelian varieties
\[f_\phi:\prod_{j}\Jac_{X/H_j'}\longrightarrow \prod_{i}\Jac_{X/H_i}\]
such that:
\begin{itemize}[leftmargin=0.6cm]
\item[(1)] (duality): denoting by $\phi^\vee: \mathbb{Z}[S']\rightarrow \mathbb{Z}[S]$ the dual homomorphism,  we have 
\[f_{\phi^\vee}=f_\phi^\vee.\]
(Here $f_\phi^\vee$ is the dual of $f_\phi$ with respect to the product of the canonical principal polarisations.)
\item[(2)] (functoriality): given also $\phi':\mathbb{Z}[S']\rightarrow \mathbb{Z}[S'']$, we have \[f_{\phi'\circ \phi}=f_\phi \circ f_{\phi'}.\]
\item[(3)] (Isogeny criterion):  suppose that the map 
\[\phi^*: \bigoplus_{j}V_\ell(\JX)^{H_j'} \longrightarrow \bigoplus_{i}V_\ell(\JX)^{H_i}\]
on rational $\ell$-adic Tate modules is an isomorphism.
 Then $f_\phi$ is an isogeny.
\end{itemize}
\end{theorem}

\begin{remark} \label{rem:desmit_chen_remark}
As mentioned above, a version of of Theorem \ref{main_G_mod_corresp_thm}  appears in work of Chen  \cite{MR1779608} and de Smit and Edixhoven \cite{MR1000113}. Aside from the extra generality afforded by our relaxed notion of a curve, Theorem \ref{main_G_mod_corresp_thm} has some additional technical advantages over the maps constructed in those works. Specifically, in place of $f_\phi$, Chen uses the map denoted $\alpha(\phi)$ in Construction \ref{alpha_phi}. While adequate for his purposes, this map does not (quite) satisfy the duality and functoriality statements in parts (1) and (2) of Theorem \ref{main_G_mod_corresp_thm}. By contrast, the work of de Smit and Edixhoven uses the same map  $f_\phi$ as we do, but constructs it only in the category of abelian varieties with isogenies inverted. We are able to improve on this as a consequence of  Lemma \ref{functorial_isogeny_lemma} below.
\end{remark}

We now turn to the proof of Theorem \ref{main_G_mod_corresp_thm}, which will take the rest of the section. 

\begin{notation} \label{Notation: quasi-nice curves}
For $g\in G$, denote by $g_*$ the corresponding automorphism of $\JX$ (cf. Definition~\ref{quasi_nice_f_star}).  
For a subgroup $H$ of $G$,  denote by $\pi_H:X\rightarrow X/H$ the quotient morphism. As in \S \ref{induced_map_jac_q_n}, we have associated mutually dual homomorphisms
\[\pi_H^*:\Jac_{X/H}\longrightarrow \JX \quad \textup{ and }\quad (\pi_H)_*:\JX\longrightarrow \Jac_{X/H}.\]
We also set $\JX^H=\cap_{h\in H}\ker(h_*-1)$. This is a group subvariety of $\JX$. The connected component of the identity $(\JX^H)^0$ is an abelian subvariety of $\JX$.
\end{notation}

 \begin{lemma} \label{prop:main_input_maps} 
Let $H$ be a subgroup of $G$.  
 \begin{itemize}
 \item[(1)] We have 
 \[(\pi_H)_*\circ (\pi_H)^* =|H|\quad \textup{ and }\quad  (\pi_H)^*\circ  (\pi_H)_*=\sum_{h\in H}h_*.\]
 \item[(2)] The image of $(\pi_H)^*$ is contained in $ (\JX^H)^0 $ and the resulting homomorphism
 \begin{equation} \label{eq:connected_cmp_isog}
 (\pi_H)^*:\Jac_{X/H}\longrightarrow  (\JX^H)^0 
 \end{equation}
 is an isogeny. 
 \end{itemize}
 \end{lemma}
 
 \begin{proof}
 (1): Follows from Proposition \ref{pull_back_pushforward_prop_quasi_nice}. 
 (2):  We first show that the image of $(\pi_H)^*$ is contained in $ (\JX^H)^0$. Since $\Jac_{X/H}$ is connected, it suffices to show that the image of  $(\pi_H)^*$ is contained in $\JX^H$, for which we can assume $k=\bar{k}$. Working under this assumption, note that the map $(\pi_H)_*:\textup{Div}^0(X)\rightarrow \textup{Div}^0(X/H)$ of \eqref{f_star_in_quasi_divs} is surjective. From this and Lemma \ref{agree_with_obvious_maps_f_star} we deduce that $(\pi_H)_*:\JX\rightarrow \Jac_{X/H}$ is surjective, so that by (1) the image of $\pi_H^*$ is contained in the image of $\sum_{h\in H}h_*$. Since this latter image is contained in $\cap_{h\in H}\textup{ker}(h_*-1)$, the result follows. Finally, to show that the map in  \eqref{eq:connected_cmp_isog} is an isogeny, note that by (1)  the endomorphism $(\pi_H)^*\circ  (\pi_H)_*$ of $\JX$ restricts to multiplication by $|H|$ on $ (\JX^H)^0$. In particular, we have homomorphisms $(\pi_H)^*:\Jac_{X/H}\rightarrow  (\JX^H)^0 $ and $(\pi_H)_*: (\JX^H)^0 \rightarrow \Jac_{X/H}$ whose composition in each direction is multiplication by $|H|$. Thus both maps are isogenies.
 \end{proof}

 We now turn to constructing the map $f_\phi$ of Theorem \ref{main_G_mod_corresp_thm} in the case that $n=m=1$.

\begin{constructions} \label{alpha_phi}  
\textup{Let $H$ and $H'$ be subgroups of $G$, and let $\phi \in \textup{Hom}_G(\mathbb{Z}[G/H],\mathbb{Z}[G/H'])$. Let $\sum_{g\in G}m_gg$ be any lift of $\phi(H)$ under the projection $\mathbb{Z}[G]\rightarrow \mathbb{Z}[G/H']$.  We have an associated endomorphism of $\JX$ given by
\[\widetilde{\phi}:=\sum_{g\in G}m_gg_*.\]  
Since $H'$ acts trivially on $\JX^{H'}$, the restriction of $\widetilde{\phi}$ to $\JX^{H'}$ is independent of the choice of lift of $\phi(H)$, and maps $\JX^{H'}$ to $\JX^{H}$.  Consequently (cf. \cite[Lemma 3.3]{MR1779608}), we have a $k$-homomorphism of abelian varieties 
\begin{equation*}
\alpha(\phi):=(\pi_{H})_*\circ \widetilde{\phi}\circ \pi_{H'}^*: \Jac_{X/H'}\longrightarrow \Jac_{X/H},
\end{equation*}
which is again independent of choices.}
\end{constructions}
 
The following lemma is key to obtaining an association $\phi \mapsto f_\phi$ satisfying parts (1) and (2) of Theorem \ref{main_G_mod_corresp_thm}.

\begin{lemma} \label{functorial_isogeny_lemma}
For any $\phi\in \textup{Hom}_G(\mathbb{Z}[G/H],\mathbb{Z}[G/H'])$, there exists a unique $k$-homomorphism $f_\phi:\Jac_{X/H'}\rightarrow\Jac_{X/H}$ satisfying $|H|f_\phi=\alpha(\phi)$. Moreover, we have
\begin{equation} \label{commut_diag_with_upper_star}
\pi_{H}^*\circ f_\phi=\widetilde{\phi}\circ \pi_{H'}^*,
\end{equation}
    for any choice of $\widetilde{\phi}$ as above.
\end{lemma}

\begin{proof} 
Uniqueness is standard. To prove existence it suffices to show that $\Jac_{X/H'}[|H|]\subseteq \ker(\alpha(\phi))$ (cf. \cite[Theorem 4]{MR0282985}). For this we can assume that $k=\bar{k}$, which we henceforth do. 

In what follows we use the terminology set out in Notation \ref{notat:points_on_Jacobians} to talk about $k$-points of $\Jac_{X/H}$ (resp. $\Jac_{X/H'}$). Let $Z\in \textup{Div}^0(X/H')$, and suppose that $|H|Z=\textup{div}(\eta)$ for some unit $\eta \in k(X/H')=L^{H'}$. With $\sum_{g\in G}m_gg$ any lift of $\phi(H)$ to $\mathbb{Z}[G]$ as above, we see that the quantity 
\[\zeta=\prod_{g\in G}g(\eta)^{m_g}\]
 defines  an invertible element of $L^{H}=k(X/H)$. Using the fact that both $\pi_H^*$ and $\pi_{H'}^*$ commute with $\textup{div}$ (see Remark \ref{rmk:pull_back_div_commute_quasi_nice}),  one computes
\begin{equation*}
|H|\textup{div}(\zeta)=(\pi_{H})_*\pi_{H}^*(\textup{div}(\zeta))=(\pi_{H})_* \widetilde{\phi}\pi_{H'}^*(\textup{div}(\eta))=|H|(\pi_{H})_* \widetilde{\phi}\pi_{H'}^*(Z).
\end{equation*}
Thus $(\pi_{H})_*\widetilde{\phi} \pi_{H'}^*(Z)=\textup{div}(\zeta)$, as desired. 

Now let  $f_\phi:\Jac_{X/H'}\rightarrow \Jac_{X/H}$ be as in the statement. It remains to show that  $\pi_{H}^*\circ f_\phi=\widetilde{\phi}\circ \pi_{H'}^*.$  Since $(\pi_{H})_*\pi_{H}^*=|H|$, we have $(\pi_{H})_*\pi_{H}^*f_\phi= \alpha(\phi)$. Applying $\pi_{H}^*$ to this equation gives  
\[N_{H} (\pi_{H}^*f_\phi-\widetilde{\phi}\pi_{H'}^*)=0,\]
where we use the notation $N_H=\sum_{g\in G}h_*$.
Now both $\widetilde{\phi}\pi_{H'}^*$ and $\pi_{H}^*f_\phi$ map $\Jac_{X/H'}$ into the abelian variety $(\JX^H)^0$. Since $N_{H}$ acts as multiplication by $|H|$ on $\JX^H$, we see that 
$|H|(\pi_{H}^*f_\phi-\widetilde{\phi}\pi_{H'}^*)=0,$
 giving the result. 
\end{proof}

\begin{definition} \label{phi_star_pre_def}
Given $\phi$ in $\textup{Hom}_G(\mathbb{Z}[G/H],\mathbb{Z}[G/H'])$, we define   $f_\phi:\Jac_{X/H'}\rightarrow \Jac_{X/H}$ to be the $k$-homomorphism afforded by Lemma \ref{functorial_isogeny_lemma}.
\end{definition}

\begin{remark} \label{push_forward_unique_identity}
From the uniqueness property of $f_\phi$ we see that the association $\phi \mapsto f_\phi$ is (contravariantly) functorial. More precisely, given subgroups $H_1, H_2$, $H_3$ of $G$, and  $G$-module homomorphisms 
\[\phi_1:\mathbb{Z}[G/H_1]\rightarrow \mathbb{Z}[G/H_2]\quad \textup{ and }\quad \phi_2:\mathbb{Z}[G/H_2]\rightarrow \mathbb{Z}[G/H_3],\]
 we have 
$f_{\phi_2\circ \phi_1}=f_{\phi_1}\circ f_{\phi_2}.$
\end{remark}

As  explained in  Notation \ref{notat:G_mod_homs_invars}, associated to $\phi \in \textup{Hom}_G(\mathbb{Z}[G/H],\mathbb{Z}[G/H'])$ is the dual homomorphism  $\phi^\vee:\mathbb{Z}[G/H'] \rightarrow \mathbb{Z}[G/H]$. Similarly, associated to $f_\phi:\Jac_{X/H'}\rightarrow \Jac_{X/H}$ is the dual homomorphism $f_\phi^\vee:\Jac_{X/H}\rightarrow \Jac_{X/H'}$ defined with respect to the canonical principal polarisations on $\Jac_{X/H}$ and $\Jac_{X/H'}$.  

\begin{lemma} \label{duals_compared}
We have $f_\phi^\vee=f_{\phi^\vee}$.
\end{lemma}
 
\begin{proof}  
It will be convenient to draw on an explicit description of $\textup{Hom}_G(\mathbb{Z}[G/H],\mathbb{Z}[G/H'])$ in terms of double cosets. Specifically, we have an isomorphism 
 \begin{equation} \label{double_coset_iso}
 \mathbb{Z}[H\backslash G/H'] \stackrel{\sim}{\longrightarrow} \textup{Hom}_{G}(\mathbb{Z}[G/H], \mathbb{Z}[G/H'])
 \end{equation}
 which sends a double coset $HgH'$ to the $G$-homomorphism $\phi_{HgH'}$ determined by 
 \begin{equation*}\label{double_coset_iso_formula}
 \phi_{HgH'}(H)=\sum_{u\in H/H\cap gH'g^{-1}}ugH'.
 \end{equation*}
An elementary calculation shows that, for any $g\in G$, we have
 \begin{equation} \label{dual_double_coset_hom}
 \phi_{HgH'}^\vee=\phi_{H'g^{-1}H}.
 \end{equation}

Writing $N_{H}=\sum_{h\in H}h\in \mathbb{Z}[G]$, we note that in $\mathbb{Z}[G]$ we have the identity
 \begin{equation} \label{double_coset_left_right_decomp}
 \sum_{u\in H/H\cap gH'g^{-1}}ugN_{H'}=\sum_{t\in HgH'}t=\sum_{w\in H'/H'\cap g^{-1}Hg}N_{H}gw.
 \end{equation}
 
We now prove the result. First, by \eqref{double_coset_iso} it suffices to take $\phi=\phi_{HgH'}$  for some $g\in G$.   From \eqref{double_coset_left_right_decomp} we see that $\sum_{t\in HgH'}t$ is a lift of $|H'|\phi_{HgH'}(H)$ to $\mathbb{Z}[G]$. In particular, writing $\gamma=\sum_{t\in HgH'}t_*$,  we have 
\[|H||H'|f_{\phi_{HgH'}}=(\pi_{H})_*\circ \gamma \circ \pi_{H'}^*.\]
Dualising gives 
\[|H||H'|f_{\phi_{HgH'}}^\vee=(\pi_{H'})_*\circ \gamma^\vee \circ \pi_{H}^*.\]
Further, Remark \ref{g_and_dual_inverse_rmk} gives
\[\gamma^\vee =\sum_{t\in HgH'}(t^{-1})_*=\sum_{t\in H'g^{-1}H}t_*.\]
Again appealing to \eqref{double_coset_left_right_decomp} we see that $\sum_{t\in H'g^{-1}H}t$ is a lift of $|H|\phi_{H'g^{-1}H}(H')$ to $\mathbb{Z}[G]$. The uniqueness property of $f_{\phi}$ thus gives  
\[f_{|H| \phi_{H'g^{-1}H}}=|H|f_{\phi_{HgH'} }^\vee,\]
so that $f_{\phi_{H'g^{-1}H}}=f_{\phi_{HgH'}}^\vee$. We now conclude from \eqref{dual_double_coset_hom}. 
\end{proof}

We now define the homomorphism $f_\phi$ in the generality claimed in Theorem \ref{main_G_mod_corresp_thm} by extending additively the association of Definition \ref{phi_star_pre_def}.

\begin{definition} \label{phi_star_main_def_brauer_weak}
As in the statement of  Theorem \ref{main_G_mod_corresp_thm}, let $H_1,...,H_n$ and $H_1',...,H_m'$ be subgroups of $G$, and define $G$-sets
 \[S=\bigsqcup_{i } G/H_i\quad \textup{ and }\quad S'=\bigsqcup_{j}G/H_j'.\]
Let $\phi:\mathbb{Z}[S]\rightarrow \mathbb{Z}[S']$ be a $G$-module homomorphism.  Since $\mathbb{Z}[S]=\bigoplus_{i}\mathbb{Z}[G/H_i]$ and $\mathbb{Z}[S']=\bigoplus_{j}\mathbb{Z}[G/H_j']$, the homomorphism $\phi$ corresponds to a collection of $G$-module homomorphisms $\phi_{ij}:\mathbb{Z}[G/H_i]\rightarrow \mathbb{Z}[G/H'_j]$.  Via Definition \ref{phi_star_pre_def}   we associate to this data the collection of \hbox{$k$-homomorphisms} $(f_{\phi_{ij}})_{i,j}$. We then define $f_{\phi}$ to be the corresponding $k$-homomorphism
\begin{equation*} \label{map_of_jacobians}
f_{\phi}:\prod_{j}\Jac_{X/H_j'}\longrightarrow \prod_{i}\Jac_{X/H_i}. 
\end{equation*}
\end{definition}

\begin{proof}[Proof of Theorem \ref{main_G_mod_corresp_thm}, parts (1) and (2)]
By additivity, the statements reduce formally to the case $n=m=1$. After these reductions, part (1) is Remark \ref{push_forward_unique_identity}, and part (2) is Lemma \ref{duals_compared}. 
\end{proof}

\subsection{$G$-descent}
To prove Theorem \ref{main_G_mod_corresp_thm}, it remains to establish the isogeny criterion. For this we will make use of the following lemma. In the application,  $F$ will  be the functor sending an abelian variety to its rational $\ell$-adic Tate module for a prime $\ell$. However, it is useful to consider the following general setup.  

\begin{lemma} \label{Additive functor lemma}
Let $E$ be a field of characteristic coprime to $|G|$, and let $F$ be an additive (covariant) functor from the category of abelian varieties over $k$ to the category of finite dimensional $E$-vector spaces.   Then, for each subgroup $H$ of $G$,  $F(\pi_H^*)$ induces an isomorphism
\begin{equation*}
  F(\Jac_{X/H})\stackrel{\sim}{\longrightarrow}F(\JX)^H.
\end{equation*} 
Moreover, given $S$, $S'$ and $\phi: \mathbb{Z}[S]\rightarrow \mathbb{Z}[S']$ as in the statement of Theorem \ref{main_G_mod_corresp_thm}, the diagram 
\[
\xymatrix{ \bigoplus_{j}F(\Jac_{X/H_j'})  \ar[d]_{\oplus_{j}F(\pi_{H_j'}^*)}\ar[r]^{F(f_\phi)}&\bigoplus_{i}F(\Jac_{X/H_i})   \ar[d]^{\oplus_{i}F(\pi_{H_i}^*)}  \\ \bigoplus_{j}F(\JX)^{H_j'}\ar[r]^{\phi^*} & \bigoplus_{i}F(\JX)^{H_i}}
\]   
commutes.
\end{lemma}

\begin{proof}
For the first claim, for any subgroup $H$ of $G$, it follows from Lemma \ref{prop:main_input_maps} that the maps $F(\pi_H^*)$ and $F((\pi_H)_*)$  induce $E$-linear maps between $F(\Jac_{X/H})$ and $F(\JX)^H$
whose composition in each direction is multiplication by $|H|$. Since the characteristic of $E$ is coprime to $|H|$ by assumption, each map is an isomorphism. 

Since $F$ is additive, commutativity of the diagram formally reduces to the case where $S=G/H$ and $S'=G/H'$ for subgroups $H$ and $H'$ of $G$.  This latter statement follows from applying $F$ to \eqref{commut_diag_with_upper_star}.  
\end{proof}

\begin{proof}[Proof of Theorem \ref{main_G_mod_corresp_thm}, part (3)]
Taking $F=V_{\ell}(-)$ in Lemma \ref{Additive functor lemma}, we find that the map on rational Tate modules induced by $f_\phi$ is an isomorphism. Thus $f_\phi$ is an isogeny.
\end{proof}

\begin{remark} \label{remark:Contravariant_functor_lemma}
In  the statement of Lemma \ref{Additive functor lemma}, suppose instead that $F$ is contravariant. Then for each subgroup $H$ of $G$, $F((\pi_H)_*)$ induces an isomorphism
\begin{equation*} \label{eq:second_inv_iso_functor}
  F(\Jac_{X/H})\stackrel{\sim}{\longrightarrow}F(\JX)^H,
\end{equation*} 
and the diagram 
\[
\xymatrix{ \bigoplus_{i}F(\Jac_{X/H_i})  \ar[d]_{\oplus_{i}F((\pi_{H_i})_*)}\ar[r]^{F(f_\phi^\vee)}&\bigoplus_{j}F(\Jac_{X/H_j'})   \ar[d]^{\oplus_{j}F((\pi_{H_j'})_*)}  \\ \bigoplus_{i}F(\JX)^{H_i}\ar[r]^{(\phi^\vee)^*} & \bigoplus_{j}F(\JX)^{H_j'}}
\]   
commutes. This can be proven similarly to the covariant case. Alternatively, one can deduce it by applying Lemma \ref{Additive functor lemma} to the covariant functor $F\circ \vee$.
\end{remark}

\section{Rational points, Selmer groups and Shafarevich--Tate groups}
In this section  we prove some results concerning the arithmetic of Jacobians of  curves with an action of a finite group $G$. As usual, Convention \ref{convention} is in place. 

 Some of the results in this section hold more generally for principally polarised abelian varieties, so we expand to this generality where appropriate. The main results of the section are summarised in Theorem \ref{Galois_descent_rational_points} below, whose  proof occupies Sections \ref{sec:isotypic} and  \ref{Sec:Selmer Groups}. Section \ref{sec:height_pairings_background} gives a complement to Theorem \ref{Galois_descent_rational_points} concerning height pairings. Finally, in Sections \ref{sec:TMG} and \ref{sec:deficiency_secs}, we give some auxiliary results concerning the $G$-action on differentials, and the expected order of Shafarevich--Tate groups modulo squares. 

\begin{notation}
For an abelian variety $A$ over a number field $K$, and for an integer $n\geq 2$, we denote by $\textup{Sel}_n(A)$ the $n$-Selmer group of $A$. For a prime $p$, we write $\textup{Sel}_{p^\infty}(A)=\varinjlim_n \textup{Sel}_{p^n}(A)$ to denote the $p^{\infty}$-Selmer group of $A$. We also write   
\[\mathcal{X}_{p}(A)= \textup{Hom}_{\mathbb{Z}_{p}}(\textup{Sel}_{p^\infty}(A), \mathbb{Q}_{p}/\mathbb{Z}_{p}) \otimes_{\mathbb{Z}_{p}} \mathbb{Q}_{p}.\] 
We denote by $\Sha(A)$ the Shafarevich--Tate group of $A$.
\end{notation}

\begin{theorem} \label{Galois_descent_rational_points}
\label{Proposition: pseudo Brauer for Xp} \label{cor:selfdual}  \label{Corollary: finiteness of sha} 
Let $X$ be a curve over a number field $K$, and $G$ a finite group acting on $X$ by $K$-automorphisms. Then we have:
\begin{enumerate}
\item The quotient morphism $\pi:X \to X/G$ induces a $K$-homomorphism $\pi^{*}:\textup{Jac}_{X/G}(K) \to \textup{Jac}_{X}(K)^{G}$ with finite kernel and cokernel, and an isomorphism
\begin{equation} \label{rat_point_quot}
(\textup{Jac}_{X}(K)\otimes\mathbb{Q})^{G} \cong \textup{Jac}_{X/G}(K) \otimes \mathbb{Q}.
\end{equation}
\item Let $p$ be a prime. Then $\mathcal{X}_{p}(\textup{Jac}_{X})$ is a self-dual $\mathbb{Q}_{p}[G]$-representation, and 
\begin{equation} \label{selmer_quot}
\mathcal{X}_{p}(\JX)^{G} \cong \mathcal{X}_{p}(\Jac_{X/G}).
\end{equation}
\item If $\sha(\JX)[p^\infty]$ is finite for a prime $p$, then so is $\sha(\Jac_{X/G})[p^\infty].$ If $\sha(\JX)$ is finite then so is $\sha(\Jac_{X/G}).$
\item If $\rho$ is a $\C[G]$-representation such that $\langle  V_{\ell}(\JX), \rho \rangle =0$, then \[\langle \textup{Jac}_{X}(K) \otimes \mathbb{Q}, \rho \rangle =0.\] Moreover,  for each prime $p$, we have $\langle \mathcal{X}_{p}(\JX), \rho \rangle =0$. \newline  (Here  $\left \langle ~,~\right \rangle$ denotes the standard inner product on $\mathbb{C}[G]$-representations.)
\end{enumerate}
\end{theorem}

\begin{proof}  
(1): that $\pi^*$ induces an isomorphism as in \eqref{rat_point_quot} follows from Lemma \ref{Additive functor lemma}, which implies the rest of the part. (2): for self duality, combine Remark \ref{g_and_dual_inverse_rmk} with Theorem \ref{Theorem:Self-duality_selmer} below. The isomorphism \eqref{selmer_quot} follows from Lemma \ref{Additive functor lemma}. (3): The claim about $p$-primary parts follows from $(1)$ and $(2)$. The claim about the full Shafarevich--Tate group follows from this, after noting that  $\pi^*: \textup{Jac}_{X/G}\rightarrow \textup{Jac}_{X}$ induces an isomorphism $\Sha(\textup{Jac}_{X/G})[p^\infty]\cong \Sha(\textup{Jac}_X)^G[p^\infty]$ for each prime $p\nmid |G|$ (cf.  Proposition \ref{pull_back_pushforward_prop_quasi_nice}). 
 (4): follows from Proposition \ref{Proposition: pseudo Brauer for Xp(A)} below.
\end{proof}

\subsection{Isotypic components} \label{sec:isotypic} 
Let $A$ be an abelian variety over a number field $K$ and let $G$ be a finite group acting on $A$ by $K$-automorphisms. The following lemma shows that if a given complex representation $\rho$ of $G$ does not appear in the $\ell$-adic Tate module of $A$ for some (equivalently any) prime $\ell$, then it does not appear in various $G$-representations functorially associated to $A$.

\begin{proposition}  \label{Proposition: pseudo Brauer for Xp(A)} 
 Let $\rho$ be a $\C[G]$-representation such that $\langle   \rho, V_{\ell}(A) \rangle =0$ for some prime $\ell$. Then, for each prime $p$, we have
 \[\langle \rho, \mathcal{X}_{p}(A) \rangle = \langle \rho, A(K) \otimes \mathbb{Q}  \rangle =0.\]
\end{proposition}

\begin{proof} $V_{\ell}(A)$ has rational character by \cite[Theorem 1, Chapter 25]{milne1998lectures}. We denote by $F_{\rho}$  the field obtained by extending $\mathbb{Q}$ by the values of the characters of $\rho$. Then, a positive power of the representation $\bigoplus_{\sigma \in \textup{Gal}(F_{\rho}/\mathbb{Q})} \rho^{\sigma}$ is realisable over $\mathbb{Q}$ and shares no common constituent with $V_{\ell}(A)$. It therefore suffices to consider the case where $\rho$ is realisable over $\mathbb{Q}$. Denote by $e_{\rho}$  the idempotent in $\mathbb{Q}[G]$ which projects onto the $\rho$-isotypic part. Clearing denominators, there exists a positive integer $n$ for which $n e_{\rho}$ lies in $\mathbb{Z}[G]$. Viewed as an endomorphism of $A$, $n e_{\rho}$ induces the zero map on $T_{\ell}(A)$. In view of \cite[Proposition 12.2]{MR861974}, $n e_{\rho}$ is the zero endomorphism on $A$, and therefore induces the zero map on $\mathcal{X}_{p}(A)$ and $A(K) \otimes \mathbb{Q}$.
\end{proof}

\begin{remark}
Suppose that  $A$ is an abelian variety over a field $k$ of characteristic $0$, equipped with the action of a finite group $G$. The same proof shows that, if $\langle \rho, V_{\ell}(A)\rangle =0$ for some {$\mathbb{C}[G]$-representation} $\rho$, then $\langle \rho , \Omega^1(A) \rangle =0$ where $\Omega^{1}(A)$ is the $k$-vector space of regular differentials.
\end{remark}

\subsection{Self-duality of Selmer groups}
\label{Sec:Selmer Groups}

Let $K$ be a number field. In what follows, the pair $(A, \lambda)$ will denote an abelian variety $A/K$, equipped with a principal polarisation $\lambda$ defined over $K$. We denote by $\textup{Aut}(A,\lambda)$ the group of $K$-automorphisms $g$ of $A$ preserving the polarisation (i.e.~for~which $g^{\vee} =\lambda g^{-1} \lambda^{-1}$). 
The proof of the following result is modeled on \cite[Theorem 2.1]{dokchitser_dokchitser_2009}.
\begin{theorem}
\label{Theorem:Self-duality_selmer} Let $(A,\lambda)$ be a principally polarised abelian variety, and let $G$ be a subgroup of $\textup{Aut}(A,\lambda)$. Then $\mathcal{X}_{p}(A)$ is self-dual as a $\mathbb{Q}_{p}[G]$-representation.
\end{theorem} 

\begin{proof} Following \cite{MR2680426}, given a $K$-isogeny of abelian varieties $f:A \to B$, we write 
\begin{equation*}  \begin{split} Q(f) = \#\text{coker}(f:A(K)/A(K)_{\text{tors}} \to B(K)/B(K)_{\text{tors}} )
\cdot  \# \text{ker}(f: \Sha(A)_{\text{div}} \to \Sha(B)_{\text{div}}),\end{split} \end{equation*}
where $\Sha_{\text{div}}$ denotes the divisible part of $\Sha.$  We decompose $\mathcal{X}_{p}(A)$ into
$\mathbb{Q}_{p}$-irreducible representations, and decompose the corresponding $\mathbb{Z}_{p}$-lattice 
\[X_{p}(A)= \textup{Hom}_{\mathbb{Z}_{p}}(\varinjlim_{n} \textup{Sel}_{p^n}(A), \mathbb{Q}_{p}/\mathbb{Z}_{p})\] into $G$-stable $\mathbb{Z}_{p}$-sublattices such that \[ \mathcal{X}_{p}(A) \cong \bigoplus_{i} \tau_{i}^{n_{\tau_{i}}}, \quad X_{p}(A) \cong \bigoplus_{i} \Lambda_{\tau_i} \quad \textup{and} \quad\Lambda_{\tau_i} \otimes_{\mathbb{Z}_{p}} \mathbb{Q}_{p} \cong  \tau_{i}^{n_{\tau_{i}}},\]  where the $\tau_{i}$ are distinct $\mathbb{Q}_{p}$-irreducible representations of $G$.
We construct a self-isogeny of $A$ with a specified action on the sublattices $\Lambda_{\tau_{i}}$ of $X_{p}(A)$. For an irreducible $\mathbb{Q}_{p}[G]$-representation $\tau \in \{\tau_{i}\}$, the operator ${\textup{dim}(\tau)} \sum_{g \in G} \textup{Tr}(\tau(g))g^{-1}) \in \mathbb{Z}_{p}[G]$ kills every irreducible $G$-constituent $\tau'$ of $\mathcal{X}_{p}$ not isomorphic to $\tau$, and acts via multiplication by $\langle \tau, \tau \rangle \cdot |G|$ on every constituent of $\mathcal{X}_{p}$ which is isomorphic to $\tau$. We define
 \[z_{\tau}:=|G|\cdot \langle \tau , \tau \rangle + ({p-1}) \cdot \textup{dim}(\tau)\cdot \sum_{g \in G} \text{Tr}(\tau(g))g^{-1} \in \mathbb{Z}_{p}[G]. \] 
Elements in a small neighbourhood $U_{\tau}$ of $z_{\tau}$ in $\mathbb{Z}_{p}[G]$  act via multiplication by $p\cdot |G| \cdot \langle \tau, \tau \rangle$ composed with some isomorphism on $\Lambda_{\tau},$ and via multiplication by $|G|\cdot \langle \tau, \tau \rangle$ composed with some isomorphism on $\Lambda_{\tau'}$ when $\tau' \neq \tau$. Similarly, there's a corresponding neighbourhood $U_{\tau ^{*}}$ of $z_{\tau^{*}}$, where $\tau^{*}$ is the dual representation. Since the inversion map $g \mapsto g^{-1}$ is continuous in $\mathbb{Q}_{p}[G]$ and maps $z_{\tau}$ to $z_{\tau^{*}}$, we can find $\sigma =\sum_{g \in G} x_{g}g\in \mathbb{Z}[G] \cap U_{\tau}$ such that $\sum_{g\in G} x_{g} g^{-1}$ lies in $ U_{\tau^{*}}.$ We write $\phi_{\sigma}$ to denote the endomorphism of $A$ determined by $\sigma$. Since $G$ is a subgroup of $\textup{Aut}(A,\lambda)$, then $ \sum_{g\in G} x_{g} g^{-1}$ corresponds to the dual isogeny $\phi_{\sigma}^{\vee}$. By considering their induced actions on the $\Lambda_{\tau_{i}}$, we deduce that 
\begin{equation*} Q(\phi_{\sigma}) = (|G|\cdot \langle \tau,\tau \rangle)^{\rk_p A}\cdot p^{n_{\tau} \, \text{dim}(\tau)}\quad\textup{ and }\quad {Q}(\phi_{\sigma}^{\vee}) =(|G|\cdot \langle\tau^{*}, \tau^{*} \rangle)^{\rk_p A}\cdot p^{n_{\tau^{*}} \, \text{dim}(\tau^{*})}. \end{equation*}
Since $\langle {\tau},\tau \rangle= \langle {\tau}^{*},\tau^{*} \rangle, \textup{dim}(\tau)=\textup{dim}(\tau^{*})$ and ${Q}(\phi_{\sigma})={Q}(\phi_{\sigma}^{\vee})$  (see \cite[Theorem 4.3]{MR2680426}), we deduce that $n_{\tau}= n_{\tau^{*}}$. This holds for all $\tau \in \{\tau_{i}\}$, and so the result follows. 
\end{proof}

\subsection{Height pairings}  \label{sec:height_pairings_background}
Let $K$ be a number field and $A/K$ an abelian variety. For a line bundle $\mathcal{L}$ on $A$ we denote by $h_\mathcal{L}$ the height function associated to $\mathcal{L}$ (cf. \cite[Chapter 5, Theorem 3.2]{MR715605}). We denote by $h_A:A(k)\times A^\vee(k)\rightarrow \mathbb{R}$ the Neron--Tate height on $A\times A^\vee$. It is the height function associated to the Poincar\'{e} line bundle $\mathcal{P}_A$ on $A\times A^\vee$.    Given a polarisation $\lambda:A\rightarrow A^\vee$, we denote by $h_\lambda$ the pairing $A(k)\times A(k)\rightarrow \mathbb{R}$ defined by 
\[h_\lambda(a,b)=h_A(a,\lambda(b)).\] 
The following lemma is standard, but we include it for lack of a suitable reference. 

\begin{lemma} \label{lem:height_pairing_adjoints}
Let $\phi:A\rightarrow B$ be a homomorphism of abelian varieties over $k$. Then, for any $a\in A(k)$ and $b\in B^\vee(k)$, we have 
\[h_B(\phi(a),b)=h_A(a,\phi^\vee(b)).\]
\end{lemma}

\begin{proof}
 By \cite[Chapter 5, Proposition 3.3]{MR715605}, the composition 
 \[A\times B^\vee \stackrel{\phi\times 1}{\longrightarrow} B\times B^\vee \stackrel{h_B}{\longrightarrow} \mathbb{R} \]
 is equal to the height function $h_{(\phi\times 1)^*\mathcal{P}_B}$. Similarly, 
$h_A\circ (1\times \phi^\vee)=h_{(1\times \phi^\vee)^*\mathcal{P}_A}.$
It thus suffices to show that the line bundles $(\phi\times 1)^*\mathcal{P}_B$ and $(1\times \phi^\vee)^*\mathcal{P}_A$ agree in $\textup{Pic}(A\times B^\vee)$, which is an immediate consequence of the defining property of $\phi^\vee$ (see e.g. \cite[\S11]{MR861974}).
\end{proof}

With Lemma \ref{lem:height_pairing_adjoints} in hand we can establish the basic compatibility with heights satisfied by the map $\pi^*$ of Theorem \ref{Galois_descent_rational_points}(1).

\begin{proposition} \label{prop:heights_rescaling}
Let $X$ be a curve over a number field $K$, and $G$ a finite subgroup of $K$-automorphisms of $X$. Let $\pi:X\rightarrow X/G$ denote the quotient morphism. Then, for all $a,b$ in $\textup{Jac}_{X/G}(K)$, we have  
\[h_{\textup{Jac}_X}\left(\pi^*(a),\pi^*(b)\right)=|G| \cdot h_{\textup{Jac}_{X/G}}(a,b).\]
Here the height pairings are defined with respect to the canonical principal polarisations of Lemma \ref{is_an_ab_var}. 
\end{proposition}

\begin{proof}
This follows from Lemma \ref{lem:height_pairing_adjoints}, recalling that $\pi^*$ is dual to $\pi_{*}$ and that $\pi_{*}\circ \pi^*=|G|$ (the first of these facts holds by definition, see Definition \ref{quasi_nice_f_star}, while the second is Proposition \ref{pull_back_pushforward_prop_quasi_nice}).
\end{proof}

 \subsection{Tate modules and differentials} \label{sec:TMG}

With a view to applications, we establish some basic auxiliary results concerning the action of a finite group $G$ of automorphisms on regular differentials and on Tate modules. In the following subsection, $K$ denotes either a number field or the completion of one.    Let $A/K$ be an abelian variety.  

\begin{lemma}  \label{lemma: Properties of Omega1}  
Let $G$ be a finite group acting on $A$ via $K$-automorphisms. Then 
\begin{enumerate}
\item for any prime $\ell$, the $G$-representations  
\[V_{\ell}(A)\quad \textup{ and }\quad  \Omega^{1}(A) \oplus \Omega^{1}(A)^*\]
are isomorphic  after extending scalars to $\mathbb{C}$.
\item if   $\Omega^{1}(A)$ is self-dual  as a $G$-representation, then $\Omega^{1}(A)^{\oplus 2}$ is realisable over $\mathbb{Q}$. 
\item if $K=\mathbb{R}$, or $K$ is a number field with a real place, then $\Omega^{1}(A)$ is realisable over $\mathbb{Q}$.
\end{enumerate}
\end{lemma}

\begin{proof}
In what follows, for a field extension $F/K$, we will write $\Omega^1(A/F)$ for the $F$-vector space of regular differentials on $A_F$.

(1). Base-changing along a field embedding $\sigma : K \hookrightarrow \mathbb{C}$, it suffices to prove the corresponding result for $K$ replaced by $\mathbb{C}$ and $A$   a complex torus. In this setting, the map that sends \hbox{$\gamma \in H_{1}(A(\mathbb{C}), \mathbb{Z})$} to the map $\omega \mapsto \int_{\gamma} \omega$ induces an isomorphism of $\mathbb{R}[G]$-modules 
\begin{equation} \label{Equation: isomorphism H1 and dual of differentials}
H_{1}(A(\mathbb{C}),\mathbb{Z}) \otimes_{\mathbb{Z}} \mathbb{R} \xrightarrow{\sim} \Omega^{1}(A/\mathbb{C})^{*}.
\end{equation}
Here $\Omega^{1}(A/\mathbb{C})^{*}$ is the dual of $\Omega^{1}(A/\mathbb{C})$ as a complex representation, viewed as an $\mathbb{R}[G]$-module by forgetting the $\mathbb{C}$-vector space structure. Note that for any finite dimensional complex $G$-representation $V$, we have ${V} \otimes_{\mathbb{R}} \mathbb{C} \cong V \oplus V^{*}$.  From this we see that, as $\mathbb{C}[G]$-representations, we have
\begin{equation} \label{omega_1_blah}
H_{1}(A(\mathbb{C}),\mathbb{Z}) \otimes_{\mathbb{Z}} \mathbb{C} \cong \Omega^{1}(A/\mathbb{C}) \oplus \Omega^{1}(A/\mathbb{C})^{*}. \end{equation}
The result now follows since, for each prime $\ell$, we have an isomorphism of $\mathbb{Q}_\ell[G]$-representations
\begin{equation*}\label{Tate_and_diffs_eq}
V_{\ell}(A)\cong H_1(A(\mathbb{C}),\mathbb{Z})\otimes_{\mathbb{Z}}\mathbb{Q}_\ell.
\end{equation*}

(2). From \eqref{omega_1_blah} we see that $\Omega^{1}(A) \oplus \Omega^{1}(A)^*$ is realisable over $\mathbb{Q}$, giving the result.

(3). Under the assumption of the statement, we can choose the embedding $\sigma$ appearing in the proof of (1) so that it factors through $\mathbb{R}$. Then, each side of \eqref{Equation: isomorphism H1 and dual of differentials}  carries a natural action of complex conjugation, and the isomorphism of \eqref{Equation: isomorphism H1 and dual of differentials} is equivariant for this action. Denoting by $H_{1}(A(\mathbb{C}),\mathbb{Z})^{+}$ the $G$-submodule of $H_{1}(A(\mathbb{C}),\mathbb{Z})$ invariant under this action, we obtain from (\ref{Equation: isomorphism H1 and dual of differentials})  an isomorphism of $\mathbb{R}[G]$-modules 
\[ H_{1}(A(\mathbb{C}),\mathbb{Z})^{+} \otimes_{\mathbb{Z}} \mathbb{R} \xrightarrow{\sim} \Omega^{1}(A/\mathbb{R})^{*}. \]
Thus $\Omega^{1}(A/\mathbb{R})^{*}$  is realisable over $\mathbb{Q}$, from which the result follows. 
\end{proof}

\subsection{Deficiency and the Shafarevich--Tate group}
 \label{sec:deficiency_secs}
The aim of this subsection is to extend certain notions introduced by Poonen--Stoll \cite{poonen1999cassels} from geometrically connected curves to the curves of Convention \ref{convention}.

 \begin{definition} \label{def:deficiency} 
 Let $\mathcal{K}$ be a local field of characteristic $0$. We define an invariant $\mu_\mathcal{K}(X)\in \mathbb{Q}^{\times}$ for each curve $X$ over $\mathcal{K}$ as follows. 
 
 First, recall from \cite[\S 8]{poonen1999cassels} that a geometrically connected  curve  $X$ over $\mathcal{K}$, of genus $g$,   is called  {\em deficient} if it has no $\mathcal{K}$-rational divisor of degree $g-1$. For such $X$, we define 
 \[\mu_\mathcal{K}(X)= \begin{cases} 
2 & \text{if $X$ is deficient,} \\
1 & \textup{otherwise}.
\end{cases} \]
 Next, suppose that $X$ is connected (but not necessarily geometrically connected). Let $\mathcal{K}(X)$ denote the function field of $X$, let $\mathcal{L}$ denote the algebraic closure of $\mathcal{K}$ in $\mathcal{K}(X)$, and let $Y$ denote the geometrically connected  curve over $\mathcal{L}$ corresponding to $\mathcal{K}(X)$ via Proposition \ref{curves_vs_function_fields} (cf. also Remark \ref{nice_vs_nice_and_geom_irred}). We then define $\mu_{\mathcal{K}}(X)=\mu_\mathcal{L}(Y)$. Finally, writing $X=\bigsqcup_iX_i$ as a disjoint union of connected components, we define 
 \[\mu_\mathcal{K}(X)=\prod_{i}\mu_{\mathcal K}(X_{i}).\]
If the field $\mathcal{K}$ is clear from context, we will often omit it from the notation.
Further, when $ \mathcal{K} = K_v$ is the completion of a number field $K$ at a place $v$, we write $\mu_v(X)$ in place of $\mu_{K_v}(X)$. 
\end{definition}

\begin{remark}
Let $X$ be a connected curve over $\mathcal{K}$, and let $X_0$ be any geometric connected component of $X$. Letting $\mathcal{K}_0/\mathcal{K}$ be the smallest field extension of $\mathcal{K}$ over which $X_0$ is defined, one checks straight from the definition that $\mu_\mathcal{K}(X)=\mu_{ \mathcal{K}_0}(X_0)$. 
\end{remark}

The following is a generalisation of a result of Poonen--Stoll \cite[Theoerem 8, Corollary 12]{poonen1999cassels}, who prove it for geometrically connected curves.  

\begin{proposition} \label{Lemma: deficiency for quasi-nice curves and 2 part of Sha} Let $X$ be a curve defined over a number field $K$, and write $\sha_{0}(\JX)$ for the quotient of $\sha(\JX)$ by its maximal divisible subgroup. Then
\[\#\sha_{0}(\JX)[2^{\infty}] \equiv \prod_{\textup{$v$ place of $K$}} \mu_v(X) \mod \mathbb{Q}^{\times 2}. \] 
\end{proposition}
\begin{proof}  
Writing $X=\bigsqcup X_i$ as a disjoint union of connected components, we have $\JX=\prod_i \Jac_{X_i}$. Using this, we reduce to the case that $X$ is connected. Let $K(X)$ denote the function field of $X$, let $L/K$ denote the algebraic closure of $K$ in $K(X)$, and let $Y$ denote the geometrically connected curve over $L$ corresponding to $K(X)$ via Proposition \ref{curves_vs_function_fields}. As in Remark \ref{nice_vs_nice_and_geom_irred}, $X$ is isomorphic to the $K$-scheme obtained from $Y$ by forgetting the $L$-structure. Observe also that, for each place $v$ of $K$, we have $L\otimes_KK_v=\prod_{w\mid v}L_w$, where the product is taken over each place $w$ of $L$ dividing $v$. Using these two observations, we compute
\begin{equation*}
X\times_KK_v=(Y\times_LL)\times_K K_v=\bigsqcup_{w\mid v}Y\times_{L}L_w.
\end{equation*}
Since each $Y\times_{L}L_w$ is a geometrically connected $L_w$-curve, we conclude that 
\begin{equation}\label{prod_over_places_deficiency}
\mu_v(X)=\prod_{w\mid v}\mu_{L_w}(Y).
\end{equation}

On the other hand, Lemma \ref{lem:curve_Weil restriction} gives
$\JX= \textup{Res}_{L/K}\Jac_{Y},$
where $\textup{Res}_{L/K}$ denotes Weil restriction from $L$ to $K$.   Shapiro's lemma  then gives $\sha_0(\JX)[2^\infty]\cong \Sha_0(\Jac_{Y}/L)[2^\infty]$ \hbox{(cf. \cite[proof of Theorem 1]{MR0330174}).} Combined with \eqref{prod_over_places_deficiency}, this reduces the result to the corresponding one for the geometrically connected curve $Y/L$, which is \cite[Theorem 8, Corollary 12]{poonen1999cassels}.
\end{proof}

\end{document}